\documentclass[11pt,final]{amsart}

\usepackage{amssymb,amsmath,latexsym}
\usepackage{amsthm}
\usepackage{amsfonts}
\usepackage{enumerate}
\usepackage{booktabs}
\usepackage{mathrsfs}

\usepackage{manfnt}

\usepackage{hyperref}
\usepackage{xcolor}
\hypersetup{
    colorlinks, linkcolor={red!80!black},
    citecolor={blue!80!black}, urlcolor={blue}
}
\usepackage[color]{showkeys}

\usepackage{graphicx,pgf}
\usepackage{float}

\usepackage{fancyhdr}

\newtheorem{thm}{Theorem}[section]

\newtheorem{defn}{Definition}[section]
\newtheorem{prop}{Proposition}[section]
\newtheorem{lemma}{Lemma}[section]
\newtheorem{rmk}{Remark}[section]

\newcommand{\aiminset}[1]{\left\{#1 \right\}}

\newcommand{\aiminabs}[1]{\left\lvert #1 \right\rvert }
\newcommand{\aiminnorm}[1]{\left\| #1 \right\| }
\newcommand{\aimininner}[2]{\langle #1, #2 \rangle}

\newcommand{\aiminRmnum}[1]{ \uppercase\expandafter{\romannumeral  #1}}

\numberwithin{equation}{section}
\numberwithin{figure}{section}

\let\oldtocsection=\tocsection
\let\oldtocsubsection=\tocsubsection
\let\oldtocsubsubsection=\tocsubsubsection
\renewcommand{\tocsection}[2]{\hspace{0em}\oldtocsection{#1}{#2}}
\renewcommand{\tocsubsection}[2]{\hspace{2em}\oldtocsubsection{#1}{#2}}
\renewcommand{\tocsubsubsection}[2]{\hspace{4em}\oldtocsubsubsection{#1}{#2}}

\begin{document}
\title{The linear hyperbolic initial and boundary value problems in a domain with corners}
\author{Aimin Huang}
\author{Roger Temam}
\address{The Institute for Scientific Computing and Applied Mathematics, Indiana University, 831 East Third Street, Rawles Hall, Bloomington, Indiana 47405, U.S.A.}
\email{AH:aimhuang@indiana.edu}
\email{RT:temam@indiana.edu}
\date{\today}

\begin{abstract}
In this article, we consider linear hyperbolic Initial and Boundary Value Problems (IBVP) in a rectangle (or possibly curvilinear polygonal domains) in both the constant and variable coefficients cases. We use semigroup method instead of Fourier analysis to achieve the well-posedness of the linear hyperbolic system, and we find by diagonalization that there are only two elementary modes in the system which we call hyperbolic and elliptic modes. The hyperbolic system in consideration is either symmetric or Friedrichs-symmetrizable.
\end{abstract}

\maketitle

\setcounter{tocdepth}{2}
\tableofcontents
\addtocontents{toc}{~\hfill\textbf{Page}\par}

\section{Introduction}
The aim of this article is to investigate the well-posedness of initial and boundary value problems for linear hyperbolic systems in a domain with corners, and the main difficulty here is the choice of the boundary conditions. General hyperbolic systems in a smooth domain have been extensively studied; see e.g. the book\cite{BS07}, where the boundary conditions satisfy the Uniform Kreiss-Lopatinskii condition (UKL), see \cite{Kre70,Lop70}. 
Linear scalar hyperbolic equation in general (non-smooth) domains has been studied in \cite[Chapter V]{GR96}, and linear hyperbolic systems in regions with corners is a subject of mathematical concern since the seminal works \cite{Osh73, Osh74} which show the possible occurrence of major singularities in the corners for certain choices of the boundary conditions, and in \cite{Sar77}, where the author investigated symmetrizable systems in regions with corners and proved that the weak solution of the problem is equal to the strong solution, and in \cite{Tan78,KT80}, where the authors studied the wave equation, which can be transformed into a linear symmetric hyperbolic system, in a domain with a corner; in the last article the space by the authors used is not the usual $L^2$-space.  In the present article, we shall restrict our attention to linear hyperbolic systems 
\begin{equation}\label{eq1.0.1}
  u_t + A_1 u_x + A_2 u_y = f, 
\end{equation}
which are either symmetric or Friedrichs-symmetrizable (see \cite[Chapter 1]{BS07}), and the main results are about the well-posedness of system \eqref{eq1.0.1} supplemented with suitable boundary conditions; see Theorem~\ref{thm2.4.2} in the constant coefficients case and Subsection \ref{subsec3.3} in the variable coefficients case.

The major difference between our work and the previous works is in the way to impose the boundary conditions. In \cite{Osh73,Osh74,Tan78,KT80}, the authors imposed the boundary conditions for the two directions $Ox, Oy$ separately. 
In an earlier work \cite{HT12} we studied the case of the inviscid linearized shallow water equations with constant coefficients in a rectangle, and we found that there are essentially two different modes  which we called hyperbolic and elliptic modes, and we showed in this article that we only have these two modes in the hyperbolic system \eqref{eq1.0.1} under the assumption that $A_1^{-1}A_2$ is diagonalizable over $\mathbb C$. 
As we will see in the proof of Theorem \ref{thm2.3.1} in Appendix \ref{sec-diagonal}, the real eigenvalues of $A_1^{-1}A_2$ correspond to hyperbolic modes, while the complex eigenvalues correspond to elliptic modes.
Hence, the way to impose the boundary conditions is to specify the boundary conditions separately for these two modes.
Another difference with earlier works is that in \cite{Osh73,Osh74,Tan78,KT80}, the authors only considered the constant coefficients case for the hyperbolic system and used Fourier transform to analyze the problem, while in this article we consider both the constant and variable coefficients cases by using the semigroup method. 
In the variable coefficients case, we need to utilize the results for Beltrami equations and quasi-conformal mapping to obtain a coordinate transformation which transforms the simple elliptic system (mode) into Cauchy-Riemann equations, see the proof of Theorem \ref{thma.2} in Appendix \ref{sec-elliptic}. This is unlike the constant coefficients case, where we can explicitly write down the coordinate transformation, see the proof of \cite[Proposition 4.1]{HT12}.

This article is motivated by studies on the Local Area Models (LAMs) for the 2d inviscid shallow water equations as done in \cite{RTT08}, \cite{HT12} and is actually a continuation and generalization of \cite{HT12}. The study of LAMs in the atmosphere and oceans sciences leads to IBVP for the inviscid primitive equations in these domains. As explained in e.g. \cite{WPT97} the choice of the boundary conditions is important for the numerical simulations, as one wishes boundary conditions leading to well-posed problems (to avoid numerical explosion) and boundary conditions that are transparent, letting the waves move freely inside and outside the domain. Thanks to a vertical expansion procedure described in \cite{OS78, TT03}, the primitive equations are equivalent, in some circumstances, to a system of coupled shallow water equations; see also \cite{RTT08, HT12}.

This article is organized as follows. After this introductory section, we study the linear hyperbolic initial and boundary value problems in a rectangle in the constant and variable coefficients cases in Sections \ref{sec2} and \ref{sec3}, respectively. In both sections, we first study two simple cases (i.e. hyperbolic and elliptic cases), and then utilize the diagonalization Theorem~\ref{thm2.3.1} in Appendix~\ref{sec-diagonal} to study the full system. Section \ref{sec4} is devoted to applying the general results of Section~\ref{sec2} to some specific examples. 
Appendix~\ref{sec-diagonal} aims to prove the simultaneous diagonalization by congruence result (i.e.  Theorem~\ref{thm2.3.1}), which generalizes the diagonalization result in \cite{Uhl73}; note that in \cite{HT12} this simultaneous diagonalization is performed by explicit calculations (see the equations (1.8) and (1.12) in \cite{HT12}), and this result is essential for studying the full hyperbolic system \eqref{eq1.0.1}. In Appendix~\ref{sec-elliptic}, we consider the elliptic boundary value problem
and in the Appendix~\ref{sec-integration}, we prove an integration by parts formula, and these two appendices 
are the key ingredients for studying the model elliptic system appearing in Subsections \ref{subsec2.2},\,\ref{subsec3.2} and that we called above the elliptic mode.
Then in Appendix~\ref{sec-density}, we derive various density theorems, density of certain smooth functions in certain function spaces, which are very useful in Subsections \ref{subsec2.1},\,\ref{subsec3.1}. In Appendix~\ref{sec-semigroup}, we collect and prove some useful theorems about semigroups on Hilbert spaces.

\section{The constant coefficients case}\label{sec2}
We now begin to study the linear hyperbolic Initial and Boundary Value Problems (IBVP). Let $n$ be a positive integer. The hyperbolic partial differential equations that we consider read
\begin{equation}\label{eq2.0.1}
  u_t + A_1 u_x + A_2 u_y = f,
\end{equation}
where $u=(u_1,\cdots,u_n)^t$, $f=(f_1,\cdots,f_n)^t$, and $A_1,A_2$ are $n\times n$ real matrices. 
The system \eqref{eq2.0.1} is called \emph{Friedrichs symmetric}, or simply \emph{symmetric} if both $A_1$ and $A_2$ are symmetric; 
and the system \eqref{eq2.0.1} is called \emph{Friedrichs symmetrizable} if there exists a symmetric positive-definite matrix $S_0$ such that $S_0A_1$ and $S_0A_2$ are both symmetric. This more general case can be transformed into a \emph{symmetric} system if we use the new unknown $\tilde u:=S_0^{1/2}u$. Therefore, we only need to consider the \emph{symmetric} case.
We do not study the non-generic case where $A_1$ or $A_2$ is singular. In this section, we will develop a general theory of the IBVP for \eqref{eq2.0.1} in the constant coefficients case, and in the next section, we will consider the variable coefficients case. 

The physical domain in which we are interested is the rectangle
$$\Omega:=(0,L_1)\times (0,L_2),$$
where $L_1,L_2$ are positive constants. We denote by $\Gamma_W,\Gamma_E$ and $\Gamma_S,\Gamma_N$ the boundaries $x=0,x=L_1$ and $y=0,y=L_2$ respectively, and we let $\Gamma$ be any union of the sets $\Gamma_W,\Gamma_E,\Gamma_S,\Gamma_N$. We also write $\Omega_T=\Omega\times (0,T)$.

In order to study the system \eqref{eq2.0.1}, we will first consider two elementary cases in Subsections \ref{subsec2.1} and \ref{subsec2.2}, and then study the full system \eqref{eq2.0.1} by diagonalization result in Appendix~\ref{sec-diagonal}. The diagonalization result is essential for studying the full system \eqref{eq2.0.1}.

\subsection{The scalar (hyperbolic) case}\label{subsec2.1}
In this subsection, we start by considering the elementary scalar equation (the case when $n=1$ in \eqref{eq2.0.1})
\begin{equation}\label{eq2.1.1}
u_t + a_1 u_x + a_2 u_y = f,
\end{equation}
where $(x,y)\in\Omega$, $t\in(0,T)$, and $a_1,a_2$ are non-zero constants. We associate with \eqref{eq2.1.1} the following initial condition
\begin{equation}\label{eq2.1.2}
  u(0,x,y)=u_0(x,y),\quad\forall\,(x,y)\in\Omega.
\end{equation}
We also need to assign the boundary conditions, and we only need to choose the boundary conditions for the parts of $\partial\Omega$ corresponding to the incoming characteristics. For that reason, we have four cases to consider depending on the signs of $a_1$ and $a_2$. Here, we consider the case when $a_1,a_2>0$, and the other cases would be similar. Hence, we choose the boundary conditions at $x=0$ and $y=0$, and we consider the homogeneous case, that is
\begin{equation}\begin{split}\label{eq2.1.3}
  u=0,\quad\text{ on }\Gamma_W\cup\Gamma_S=\aiminset{x=0}\cup\aiminset{y=0}.\\
\end{split}\end{equation}
Note that the boundary conditions that we choose are \emph{strictly dissipative} (see \cite[Section 3.2]{BS07}), and hence satisfy the \emph{uniform Kreiss-Lopatinskii} condition (see e.g. \cite{Kre70,Lop70} or \cite[Chapter 4]{BS07}).
\begin{rmk}\label{rmk2.1.1}
Let us briefly present the other cases for the boundary conditions which are suitable for \eqref{eq2.1.1} based on the signs of $a_1$ and $a_2$.

\noindent If $a_1>0$ and $a_2<0$, then we choose the boundary conditions
\begin{equation}\begin{split}\label{eq2.1.3-1}
  u=0,\quad\text{ on }\Gamma_W\cup\Gamma_N=\aiminset{x=0}\cup\aiminset{y=L_2}.\\
\end{split}\end{equation}
If $a_1<0$ and $a_2>0$, then we choose the boundary conditions
\begin{equation}\begin{split}\label{eq2.1.3-2}
  u=0,\quad\text{ on }\Gamma_E\cup\Gamma_S=\aiminset{x=L_1}\cup\aiminset{y=0}.\\
\end{split}\end{equation}
If $a_1<0$ and $a_2<0$, then we choose the boundary conditions
\begin{equation}\begin{split}\label{eq2.1.3-3}
  \ u=0,\quad\text{ on }\Gamma_E\cup\Gamma_N=\aiminset{x=L_1}\cup\aiminset{y=L_2}.\\
\end{split}\end{equation}
\end{rmk}

The well-posedness of the initial and boundary value problem \eqref{eq2.1.1}-\eqref{eq2.1.3} will be proven using the classical Hille-Yosida theorem. We first define the unbounded operator $\mathcal T_1$ on $H:=L^2(\Omega)$ with $\mathcal T_1 u=a_1u_x + a_2u_y$, $\forall\,u\in\mathcal D(\mathcal T_1)$, and 
\[
\mathcal D(\mathcal T_1)=\aiminset{u\in H=L^2(\Omega)\,:\,\mathcal T_1u=a_1u_x + a_2u_y \in H,\text{ and }u\text{ satisfies }\eqref{eq2.1.3}}.
\]
If $u$ and $\mathcal T_1u$ both belong to $L^2(\Omega)$, then the traces of $u$ at the boundary are well-defined by Proposition \ref{M-propb.2}. Hence, the domain $\mathcal D(\mathcal T_1)$ is well-defined.

We now introduce the density boundary conditions corresponding to \eqref{eq2.1.3}:
\begin{equation}\label{eq2.1.3p}\tag{\ref{eq2.1.3}$'$}
  u\text{ vanishes in a neighborhood of }\Gamma_W\cup\Gamma_S,
\end{equation}
and the function space:
\[
\mathcal V(\Omega)=\aiminset{u\in\mathcal C^\infty(\overline\Omega)\,:\,u\text{ satisfies }\eqref{eq2.1.3p}}.
\]
Applying Theorem \ref{L-thm1} with $\lambda = a_1/a_2$, we obtain that 
\begin{lemma}\label{lem2.1.1}
$\mathcal V(\Omega)\cap\mathcal D(\mathcal T_1)$ is dense in $\mathcal D(\mathcal T_1)$.
\end{lemma}
\subsubsection{Positivity of $\mathcal T_1$ and its adjoint $\mathcal T_1^*$}
The proofs of the positivity of  $\mathcal T_1$ and its adjoint $\mathcal T_1^*$  in this subsection are similar or simpler to those of \cite[Subsection 3.1.1]{HT12}, and for the sake of completeness, we give the full details here.

Our aim is to prove that $\mathcal T_1$ and its adjoint $\mathcal T_1^*$ defined below are positive in the sense,
\begin{equation}\label{eq2.1.4}
\begin{cases}
\aimininner{ \mathcal T_1u}{u}_H \geq 0,\hspace{6pt}\forall\, u\in\mathcal{D}(\mathcal T_1),\\
\aimininner{ \mathcal T_1^*u}{u}_H \geq 0,\hspace{6pt}\forall\, u\in\mathcal{D}(\mathcal T_1^*).
\end{cases}
\end{equation}
These properties are needed to apply the Hille-Phillips-Yoshida theorem (see Appendix \ref{sec-semigroup}). The result for $\mathcal T_1$ is easy, and actually for $u$ smooth in $\mathcal{D}(\mathcal T_1)$, integrating by parts and using the boundary conditions \eqref{eq2.1.3} yields
\begin{equation}\begin{split}
\aimininner{ \mathcal T_1u}{u}_H &= \int_\Omega (a_1u_x+a_2u_y)u\,\text{d}x\text{d}y \\
&=\frac12\int_0^{L_2} a_1u^2\big|_{x=L_1}\,\text{d}y + \frac12\int_0^{L_1} a_2u^2\big|_{y=L_2}\,\text{d}x\\
&\geq 0.
\end{split}\end{equation}
This is also valid for all $u$ in $\mathcal{D}(\mathcal T_1)$ thanks to Lemma \ref{lem2.1.1}. Therefore, $\mathcal T_1$ is positive.

We now turn to the definition of the formal adjoint $\mathcal T_1^*$ of $\mathcal T_1$ and its domain $\mathcal{D}(\mathcal T_1^*)$, in the sense of the adjoint of a linear unbounded operator (see \cite{Rud91}). For that purpose, we assume that $u\in\mathcal{D}(\mathcal T_1)$ and $\bar u\in H=L^2(\Omega)$ are smooth functions, and then compute
\begin{equation}\begin{split}\label{eq2.1.6}
\aimininner{ \mathcal T_1u}{\bar u}_H &= \int_\Omega (a_1u_x+a_2u_y)\bar u\,\text{d}x\text{d}y \\
&=(\text{using integrations by parts and the boundary conditions }\eqref{eq2.1.3})\\
&=\int_\Omega(-a_1\bar u_x-a_2\bar u_y)u\,\text{d}x\text{d}y + \int_0^{L_2} (a_2u\bar u)\big|_{x=L_1} \,\text{d}y 
+\int_0^{L_1} (a_2u\bar u)\big|_{y=L_2} \,\text{d}x\\
&=J_0 + J_1,
\end{split}\end{equation}
where $J_0$ stands for the integral on $\Omega$ and $J_1$ for the integral on $\partial\Omega$.

Classically (see e.g. \cite{Rud91}), $\mathcal{D}(\mathcal T_1^*)$ consists of the $\bar u$ in $H=L^2(\Omega)$ such that $u\mapsto\aimininner{\mathcal T_1u}{\bar u}$ is continuous on $\mathcal{D}(\mathcal T_1)$ for the topology (norm) of $H$. If $u$ is restricted to the class of $\mathcal{C}^\infty$ functions with compact support in $\Omega$, then $J_1$ vanishes and $u\mapsto J_0$ can only be continuous if $-a_1\bar u_x-a_2\bar u_y$ belongs to $H$. If $\bar u$ and $-a_1\bar u_x-a_2\bar u_y$ both belong to $H$, the traces of $\bar u$ are well-defined on $\partial\Omega$ by  Proposition \ref{M-propb.2}. Hence, the calculations in \eqref{eq2.1.6} are now valid for any such $\bar u$ (and $u$ smooth in $\mathcal{D}(\mathcal T_1)$). We now restrict $u$ to the class of $\mathcal{C}^\infty$ functions on $\overline\Omega$ which belongs to $\mathcal{D}(\mathcal T_1)$.
Then the expressions above of $J_0$ and $J_1$ show that 
$u\mapsto\aimininner{\mathcal T_1u}{\bar u}$ can only be continuous in $u$ for the topology (norm) of $H$ if the following boundary conditions are satisfied:
\begin{equation}\label{eq2.1.7}
\bar u=0, \text{ on } \Gamma_E\cup\Gamma_N = \{x=L_1\}\cup\{y=L_2\}.
\end{equation}
We now aim to show that
\begin{equation}\label{eq2.1.8}
\mathcal{D}(\mathcal T_1^*) = \aiminset{ \bar u\in H=L^2(\Omega)\,:\,-a_1\bar u_x-a_2\bar u_y\in H, \text{ and } \bar u\text{ satisfies } \eqref{eq2.1.7}};
\end{equation}
we first conclude from the above arguments that $\mathcal{D}(T_1^*)$ is included in the space temporarily denoted by $\widetilde{\mathcal{D}}(T_1^*)$, the right-hand side of \eqref{eq2.1.8}.
To prove that $\widetilde{\mathcal{D}}(T_1^*)\subset \mathcal{D}(T_1^*)$, we only need to observe that \eqref{eq2.1.6} holds when $u\in\mathcal{D}(T_1)$ and $\bar u\in\widetilde{\mathcal{D}}(T_1^*)$ in which case \eqref{eq2.1.6} reduces to $\aimininner{\mathcal T_1u}{\bar u}=\aimininner{u}{-a_1\bar u_x-a_2\bar u_y}$. This is proved by approximation observing that the smooth functions are dense in $\mathcal{D}(T_1)$ and in $\widetilde{\mathcal{D}}(T_1^*)$ respectively. The former density result has already been proven (see Lemma \ref{lem2.1.1}), and the later one is the object of Lemma \ref{lem2.1.2} below.
Hence, if $\bar u\in \widetilde{\mathcal{D}}(T_1^*)$, then the calculations \eqref{eq2.1.6} are valid, $J_1$ vanishes, and $u\mapsto\aimininner{\mathcal T_1u}{\bar u}$ is continuous on $\mathcal{D}(\mathcal T_1)$ for the norm of $H$. Therefore we conclude that $\widetilde{\mathcal{D}}(T_1^*)\subset \mathcal{D}(T_1^*)$ and thus \eqref{eq2.1.8} holds. We then set $\mathcal T_1^*\bar u=-a_1\bar u_x-a_2\bar u_y$, $\forall\,\bar u\in\mathcal{D}(\mathcal T_1^*)$.

Let us introduce the density boundary conditions corresponding to \eqref{eq2.1.7}
\begin{equation}\label{eq2.1.7p}\tag{\ref{eq2.1.7}$'$}
\bar u \text{ vanishes in a neighborhood of }\Gamma_E\cup\Gamma_N,
\end{equation}
and define the following space of smooth functions
\begin{equation*}
\mathcal{V}^*(\Omega) = \{ \bar u\in \mathcal{C}^\infty(\overline{\Omega})\,:\,\bar u \text{ satisfies \eqref{eq2.1.7p}} \}.
\end{equation*}
Applying Theorem \ref{L-thm1} and Remark \ref{L-rmk2}, we have
\begin{lemma}\label{lem2.1.1-2}
$\mathcal{V}^*(\Omega)\cap\widetilde{\mathcal{D}}(\mathcal T_1^*)$ is dense in $\widetilde{\mathcal{D}}(\mathcal T_1^*)$.
\end{lemma}
Since we derived that $\widetilde{\mathcal{D}}(T_1^*) = \mathcal{D}(T_1^*)$, Lemma \ref{lem2.1.1-2} shows that 
\begin{lemma}\label{lem2.1.2}
$\mathcal{V}^*(\Omega)\cap\mathcal{D}(\mathcal T_1^*)$ is dense in $\mathcal{D}(\mathcal T_1^*)$.
\end{lemma}
The proof of the positivity of $\mathcal T_1^*$ is similar to the proof for $\mathcal T_1$ using Lemma \ref{lem2.1.2}, we thus omit it here.

\subsubsection{Semigroup} We are now ready to prove the main theorem in this subsection.
\begin{thm}\label{thm2.1.1}
The operator $-\mathcal T_1$ is the infinitesimal generator of a contraction semigroup on $H=L^2(\Omega)$.
\end{thm}

\begin{proof}
According to Theorem \ref{S-thm2.2}, it suffices to show that
\begin{enumerate}[(i)]
\item $\mathcal T_1$ and $\mathcal T_1^*$ are both closed, and their domains $\mathcal{D}(\mathcal T_1)$ and $\mathcal{D}(\mathcal T_1^*)$ are both dense in $H$.
\item $\mathcal T_1$ and $\mathcal T_1^*$ are both positive.
\end{enumerate}
We already proved that $\mathcal T_1$ and $\mathcal T_1^*$ are both positive, we thus only need to prove (i). 
We establish the result for $\mathcal T_1$, and the proof for $\mathcal T_1^*$ would be similar.

Observing that $\mathcal{D}(\Omega)$ is included in $\mathcal{D}(\mathcal T_1)$ and dense in $H=L^2(\Omega)$, we thus obtain $\mathcal{D}(\mathcal T_1)$ is dense in $H$. 
To show that $\mathcal T_1$ is closed. Consider a sequence $\{u_n\}_{n\in\mathbb{N}}\subset \mathcal{D}(\mathcal T_1)$ for which $\lim_{n\rightarrow\infty}u_n=u$ in $L^2(\Omega)$ and $\lim_{n\rightarrow\infty}\mathcal T_1u_n =\hat u$ in $L^2(\Omega)$.
We first have that $\partial_xu_n, \partial_yu_n$ converge to $\partial_xu, \partial_yu$ in the sense of distributions, respectively. Hence,  $\mathcal T_1u_n$ converges to $\mathcal T_1u$ in the sense of distributions, which implies that $\mathcal T_1u$ equals $\hat u$ in the sense of distributions. Therefore, we conclude that $\mathcal T_1u$ belongs to $L^2(\Omega)$ since $\hat u$ belongs to $L^2(\Omega)$.
We thus have $u, \mathcal T_1u\in L^2(\Omega)$, which shows that the traces of $u$ on $\partial\Omega$ are well-defined thanks to Proposition~\ref{M-propb.2}. By Proposition~\ref{M-propb.2} again, the traces of $u_n$ converge weakly to the traces of $u$ in the appropriate space $H^{-1}$, so that $u$ satisfies the boundary conditions \eqref{eq2.1.3}. Therefore, we conclude that $u\in\mathcal{D}(\mathcal T_1)$ and that $\mathcal T_1$ is closed. The proof is complete. 
\end{proof}

\begin{rmk}\label{rmk2.1.2}
Looking back carefully at the proof of Theorem~\ref{thm2.1.1}, we see that Theorem~\ref{thm2.1.1} is still valid if the domain $\mathcal D(\mathcal T_1)$ is properly changed when we are in one of the other cases presented in Remark~\ref{rmk2.1.1}.
\end{rmk}

\begin{rmk}\label{rmk2.1.3}
The result in Theorem~\ref{thm2.1.1} can also be extended to the curvilinear polygonal domains instead of rectangle as long as we assign the boundary conditions for the incoming characteristics. Indeed, in \cite[pp. 419-423]{GR96}, the authors provided an explicit way to find the solution for the IBVP \eqref{eq2.1.1} in curvilinear polygonal domains, and using similar arguments as we did for Theorem~\ref{thm2.1.1}, we could also obtain a semigroup structure for the operator $\mathcal T_1$ in curvilinear polygonal domains.

Here we avoid the non-generic case where the vector field $(a_1,a_2)$ is parallet to one side of the polygon. In fact this case could be treated with some precautions as in e.g. \cite{RTT08}.

\end{rmk}

With Theorem~\ref{thm2.1.1} (and Remark \ref{rmk2.1.2}) at hand, we are able to solve the initial and  boundary value problem \eqref{eq2.1.1}-\eqref{eq2.1.3} either weakly or classically under suitable assumptions on $u_0$ and $f$. We do not state these results in this subsection, but we will state these results for the full system in Subsection \ref{subsec2.4}.

\subsection{The simple (elliptic) system case}\label{subsec2.2}
In this subsection, we consider the following simple system (the case when $n=2$ in \eqref{eq2.0.1})
\begin{equation}\begin{cases}\label{eq2.2.1}
  u_t+T_1u_x + T_2u_y = f,\\
  u(0)=u_0,
\end{cases}\end{equation}
where $u=(u_1,u_2)^t$, $f=(f_1,f_2)^t$, and 
\begin{equation*}
T_1=\begin{pmatrix}
\alpha_1 & \beta_1\\
\beta_1 & -\alpha_1
\end{pmatrix},\hspace{6pt}
T_2=\begin{pmatrix}
\alpha_2 & \beta_2\\
\beta_2 & -\alpha_2
\end{pmatrix}.
\end{equation*}
Here, we assume that $\alpha_1,\alpha_2,\beta_1,\beta_2$ are four real constants satisfying
 \begin{equation}\label{asp2.2.1}
 \alpha_2\beta_1-\alpha_1\beta_2> 0.
 \end{equation}
Note that under assumption \eqref{asp2.2.1}, both $T_1$ and $T_2$ are non-singular.
 
We also need to assign the boundary conditions for \eqref{eq2.2.1}, and we choose the boundary conditions for the operator $\mathcal T_2$ defined below to be positive (see Proposition \ref{prop2.2.2}). For that reason, we have four cases to consider depending on the signs of $\alpha_1$ and $\alpha_2$. Notice that one of $\alpha_1,\alpha_2$ must be non-zero by assumption \eqref{asp2.2.1}. Here, we consider the case when $\alpha_1,\alpha_2\geq 0$, and the other cases would be similar.
Hence, we choose the following homogeneous boundary conditions
\begin{equation}\label{eq2.2.2}
\begin{cases}
 u_1 = 0 \text{ on } \Gamma=\Gamma_W\cup\Gamma_S=\aiminset{x=0}\cup\aiminset{y=0}, \\
 u_2 = 0 \text{ on } \Gamma^c=\Gamma_E\cup\Gamma_N=\aiminset{x=L_1}\cup\aiminset{y=L_2}.
 \end{cases}
\end{equation}
Here $\Gamma^c$ is the complement of $\Gamma$ with respect to the boundary $\partial\Omega$. 

Let us first give several remarks about the choice of the boundary conditions, and then we will show that the boundary conditions \eqref{eq2.2.2} will lead to a well-posedness result for \eqref{eq2.2.1}.
\begin{rmk}\label{rmk2.2.1}
Let us present the other cases of the boundary conditions which are suitable for \eqref{eq2.2.1} based on the signs of $\alpha_1$ and $\alpha_2$.

\noindent If $\alpha_1\geq 0$ and $\alpha_2<0$, then we choose the boundary conditions
\begin{equation}\label{eq2.2.2-1}
\begin{cases}
 u_1 = 0 \text{ on } \Gamma_W\cup\Gamma_N=\aiminset{x=0}\cup\aiminset{y=L_2}, \\
 u_2 = 0 \text{ on } \Gamma_E\cup\Gamma_S=\aiminset{x=L_1}\cup\aiminset{y=0}.
 \end{cases}
\end{equation}
If $\alpha_1< 0$ and $\alpha_2\geq 0$, then we choose the boundary conditions
\begin{equation}\label{eq2.2.2-2}
\begin{cases}
 u_1 = 0 \text{ on } \Gamma_W\cup\Gamma_S=\aiminset{x=L_1}\cup\aiminset{y=0}, \\
 u_2 = 0 \text{ on } \Gamma_E\cup\Gamma_N=\aiminset{x=0}\cup\aiminset{y=L_2}.
 \end{cases}
\end{equation}
If $\alpha_1< 0$ and $\alpha_2<0$, then we choose the boundary conditions
\begin{equation}\label{eq2.2.2-3}
\begin{cases}
 u_1 = 0 \text{ on } \Gamma_E\cup\Gamma_N=\aiminset{x=L_1}\cup\aiminset{y=L_2}, \\
 u_2 = 0 \text{ on } \Gamma_W\cup\Gamma_S=\aiminset{x=0}\cup\aiminset{y=0}.
 \end{cases}
\end{equation}
\end{rmk}
\begin{rmk}\label{rmk2.2.1-1}
Besides the options presented in Remark \ref{rmk2.2.1}, we can also choose other sets of boundary conditions depending on the relative signs of $\alpha_1,\beta_1$ and $\alpha_2,\beta_2$, which will lead to similar results. Indeed, we only need to consider a change of variables as follows:
\begin{equation}
v=\begin{pmatrix}
v_1\\
v_2
\end{pmatrix}=Qu:=\frac{1}{\sqrt{1+\kappa^2}}\begin{pmatrix}
\kappa&-1\\
1&\kappa
\end{pmatrix}\begin{pmatrix}
u_1\\
u_2
\end{pmatrix}=\frac{1}{\sqrt{1+\kappa^2}}\begin{pmatrix}
\kappa u_1-u_2\\
u_1+\kappa u_2
\end{pmatrix},
\end{equation}
where $\kappa$ is a non-zero constant.
Then we rewrite \eqref{eq2.2.1} in the new variable $v=(v_1,v_2)$:
\begin{equation}\begin{cases}
  v_t+\overline T_1v_x + \overline T_2v_y = \bar f,\\
  v(0)=v_0,
\end{cases}\end{equation}
where $v_0=Qu_0$, $\bar f=Qf$, and 
\begin{equation*}
\overline T_1=QT_1Q^{-1}=\begin{pmatrix}
\bar\alpha_1 & \bar\beta_1\\
\bar\beta_1 & -\bar\alpha_1
\end{pmatrix},\hspace{6pt}
\overline T_2=QT_2Q^{-1}=\begin{pmatrix}
\bar\alpha_2 & \bar\beta_2\\
\bar\beta_2 & -\bar\alpha_2
\end{pmatrix}.
\end{equation*}
Here, we have
\begin{equation}\begin{split}
  \bar\alpha_1 &=\frac{(\kappa^2-1)\alpha_1 - 2\kappa\beta_1}{1+\kappa^2},\quad
  \bar\alpha_2=\frac{(\kappa^2-1)\alpha_2 - 2\kappa\beta_2}{1+\kappa^2},\\
  \bar\beta_1&=\frac{(\kappa^2-1)\beta_1 + 2\kappa\alpha_1}{1+\kappa^2},\quad
  \bar\beta_2=\frac{(\kappa^2-1)\beta_2 + 2\kappa\alpha_2}{1+\kappa^2},
\end{split}\end{equation}
and
\[
 \bar\alpha_2\bar\beta_1-\bar\alpha_1\bar\beta_2 = \alpha_2\beta_1-\alpha_1\beta_2.
\]
Therefore, once we know the signs of $\bar\alpha_1,\bar\alpha_2$, we can choose the boundary conditions for $v$ and thus for $u$ according to \eqref{eq2.2.2} and Remark \ref{rmk2.2.1}. This remark shows that there are infinitely many sets of boundary conditions for the system \eqref{eq2.2.1} which will lead to a well-posedness result.\qed
\end{rmk}
\begin{rmk}\label{rmk2.2.1-2}
The choice of the boundary conditions is actually very flexible as long as the boundary conditions we choose will imply that the operator $\mathcal T_2$ defined below is positive (see Proposition \ref{prop2.2.2}). This flexibility comes from Theorem \ref{thma.1} and Remark \ref{rmka.1} in Appendix \ref{sec-elliptic} since the main ingredient in the proof of Proposition \ref{prop2.2.1} and Theorem \ref{thm2.2.1} below is Theorem \ref{thma.1}, which is an extension of \cite[Proposition 4.1]{HT12}. For example, if we assume that $\alpha_1,\alpha_2,\beta_1,\beta_2$ are four real positive constants that satisfy \eqref{asp2.2.1}, then we can choose the following boundary conditions:
\begin{equation}\begin{split}
u_1&=0\,\text{ on }\Gamma_W, \quad\quad u_1 - u_2 = 0\,\text{ on }\Gamma_E,\\
u_1+u_2&=0\,\text{ on }\Gamma_S, \quad\quad\quad\quad\;\, u_2 = 0\,\text{ on }\Gamma_N,
\end{split}\end{equation}
which will make the operator $\mathcal T_2$ defined below positive. \qed
\end{rmk}

We are also going to use the classical Hille-Yosida theorem as in Subsection \ref{subsec2.1} to solve the IBVP \eqref{eq2.2.1}-\eqref{eq2.2.2}, and we thus define the unbounded operator $\mathcal T_2$ on $H^2:=L^2(\Omega)^2$ by setting
\begin{equation}\label{eq2.2.3}
\mathcal T_2 u=T_1 u_x+T_2 u_y,\quad \forall\,u\in  \mathcal D(\mathcal T_2),
\end{equation}
and
\[
 \mathcal D(\mathcal T_2)=\aiminset{  u=(u_1,u_2)^t\in H^2=L^2(\Omega)^2\,:\, \mathcal T_2 u\in H^2, u\text{ satisfies \eqref{eq2.2.2} } }.
\]
In order to prove that $\mathcal T_2$ generates a semigroup on $H^2$, unlike in the Subsection \ref{subsec2.1} where we prove that both $\mathcal T_1$ and its adjoint $\mathcal T_1^*$ are positive, we prove here that $\mathcal T_2$ is positive and that $\mathcal T_2+\omega$ is onto for all $\omega>0$.

We now introduce the function space
\begin{equation*}
V=\{ u=(u_1,u_2)^t\in H^1(\Omega)^2 \ |\ u\text{ satisfies \eqref{eq2.2.2} } \}.
\end{equation*}
Note that thanks to the Poincar\'e inequality $\aiminnorm{\nabla u}_{L^2(\Omega)^2}$ is a norm on $V$ equivalent to that of $H^1(\Omega)^2$.

\subsubsection{Properties of $\mathcal T_2$}
Note that the operator $\mathcal T_2 $ is the operator $\mathcal T$ in \cite[Section 4.1]{HT12},
and we have already proven the following results in \cite[Theorems 4.1 and 4.2]{HT12}.
\begin{prop}\label{prop2.2.1}
The domain $\mathcal D(\mathcal T_2)$ of $\mathcal T_2$ is the space $V$, and the following estimate
 \begin{equation}\label{eq2.2.8}
 \frac{1}{c_1}\aiminnorm{\nabla u}_{L^2(\Omega)^2}\leq \aiminnorm{\mathcal T_2 u}_{L^2(\Omega)^2}\leq c_2 \aiminnorm{\nabla u}_{L^2(\Omega)^2},
 \end{equation}
holds for some constants $c_1,c_2>0$.
\end{prop}

\begin{prop}\label{prop2.2.2}
The operator $\mathcal T_2$ is positive, i.e. for any $u\in\mathcal D(\mathcal T_2) = V$:
\begin{equation}
  \aimininner{\mathcal T_2 u}{u}\geq 0.
\end{equation}
\end{prop}

\subsubsection{The surjectivity of $\mathcal T_2$}
It is clear that $\mathcal T_2$ is a linear continuous operator from $V$ to $L^2(\Omega)^2$ and it is one-to-one by Proposition \ref{prop2.2.1}. We can also prove that $\mathcal T_2$ is onto, and we actually prove the following result.
\begin{thm}\label{thm2.2.1}
For every given $f=(f_1,f_2)^t\in L^2(\Omega)^2$ and $\omega\geq 0$, the problem $\mathcal T_2 u + \omega u= f$ has a unique solution $u\in \mathcal D(\mathcal T_2)=V$.
\end{thm}

The proof of Theorem \ref{thm2.2.1} is the same as the proof of \cite[Theorem 4.3]{HT12} except that we need to replace the bilinear form $a$ in \cite{HT12} by
\[
a(u,\bar{u})=\aimininner{\mathcal T_2 u}{\mathcal T_2 \bar{u}} + \omega\aimininner{u}{\mathcal T_2 \bar{u}}.
\]

\subsubsection{Semigroup}
We are now ready to prove the main theorem in this subsection.
\begin{thm}\label{thm2.2.2}
The operator $-\mathcal T_2$ is the infinitesimal generator of a contraction semigroup on $H^2=L^2(\Omega)^2$.
\end{thm}
\begin{proof}
Since $\mathcal T_2$ is continuous from $V$ to $H^2$ and $V$ is dense in $H^2$, it is clear that $\mathcal T_2$ is a closed, densely defined operator on $H^2$. Combining Proposition \ref{prop2.2.2} and Theorem \ref{thm2.2.1} and applying Theorem \ref{S-thm2.1} [Hille-Yosida theorem] to $\mathcal T_2$, we obtain the result.
\end{proof}

\begin{rmk}\label{rmk2.2.2}
Looking back carefully at the proof of Theorem~\ref{thm2.2.1}, we see that Theorem~\ref{thm2.2.1} is still valid if the function space $V$ (i.e. the domain $\mathcal D(\mathcal T_2)$) is properly changed when we are in any of the other cases presented in Remarks~\ref{rmk2.2.1}-\ref{rmk2.2.1-2}.
\end{rmk}

\begin{rmk}\label{rmk2.2.3}
We could also extend the results in Theorem~\ref{thm2.2.1} to  curvilinear polygonal domains, whose boundaries are made of piecewise $\mathcal C^1$ curves. Indeed, in the proof of Theorem~\ref{thma.1} which is essential for Theorem~\ref{thm2.2.1}, the results that we use from Grisvard's book \cite{Gri85} are still valid for curvilinear polygonal domains. Arguing as for Theorem~\ref{thm2.2.1}, we could also obtain a semigroup structure for the operator $\mathcal T_2$ in the curvilinear polygonal domains case as long as the boundary conditions we assign would make the operator to be positive.
Also we observe again that, thanks to the regularity results for elliptic problems, we do not need here to prove density results as we do for the hyperbolic mode.
\end{rmk}

With Theorem~\ref{thm2.2.1} (and Remark \ref{rmk2.2.2}) at hand, we are able to solve the initial and  boundary value problem \eqref{eq2.2.1}-\eqref{eq2.2.2} either weakly or classically under suitable assumptions on $u_0$ and $f$. As before in Subsection \ref{subsec2.1}, we do not state these results in the present subsection and refer the interested reader to the general result in Subsection \ref{subsec2.4} (Theorem \ref{thm2.4.2}).

\subsection{The full system}\label{subsec2.4}
With the diagonalization Theorem \ref{thm2.3.1} presented in Appendix~\ref{sec-diagonal}, we are able to decompose the full system \eqref{eq2.0.1} into  simple equations which are either hyperbolic modes or elliptic modes according to the terminology above. The full system reads
\begin{equation}\begin{cases}\label{eq2.4.1}
u_t + A_1 u_x + A_2 u_y = f,\\
u(0)=u_0,\\
u\text{ satisfies \emph{suitable boundary conditions}},
\end{cases}\end{equation}
where $u=(u_1,\cdots,u_n)^t$, $f=(f_1,\cdots,f_n)^t$, and $A_1,A_2$ are real non-singular symmetric $n\times n$ matrices, and the \emph{suitable boundary conditions} will be explained below. We suppose that $A_1, A_2$ satisfy the assumptions in Theorem \ref{thm2.3.1}. 
By Theorem \ref{thm2.3.1}, we know that there exists a non-singular matrix $P$ which can diagonalize $A_1$ and $A_2$ simultaneously, i.e.
\begin{equation}\label{eq2.4.2}
  \begin{split}
P^tA_1P=\bar A_1=\text{diag}(C_1,\cdots,C_m),\\
P^tA_2P=\bar A_2=\text{diag}(D_1,\cdots,D_m),\\  
  \end{split}
\end{equation}
for some integer $m$ satisfying $1\leq m\leq n$, where the pair $(C_i,D_i)$ ($i=1,\cdots,m$) is either of \emph{Type I} or of \emph{Standard Type II} (see the definitions in Appendix \ref{sec-diagonal}).
We now define the unbounded operator $\bar A$ on $H^n=L^2(\Omega)^n$ with $\bar A\bar u=\bar A_1\bar u_x+\bar A_2\bar u_y$, $\forall\,\bar u\in \mathcal D(\bar A)$, and
\begin{equation}\label{eq2.4.3}
\mathcal D(\bar A)=\aiminset{\bar u\in H^n\,:\,\bar A\bar u\in H^n,\, \bar u\text{ satisfies \emph{suitable boundary conditions}}}.
\end{equation}
Let us explain how we choose the \emph{suitable boundary conditions} for $\bar u$ in the domain $\mathcal D(\bar A)$. If the pair $(C_1,D_1)$ is of \emph{Type I},
then we choose the boundary condition for $\bar u_1$ according to \eqref{eq2.1.3} and Remark \ref{rmk2.1.1} depending on the signs of $C_1$ and $D_1$; while if the pair $(C_1,D_1)$ is of \emph{Type II}, i.e. the pair $(C_1,D_1)$ is of form 
\begin{equation}
\bigg(  \begin{pmatrix}
\alpha_1 & \beta_1\\
\beta_1 & -\alpha_1
\end{pmatrix},
  \begin{pmatrix}
\alpha_2 & \beta_2\\
\beta_2 & -\alpha_2
\end{pmatrix}
\bigg), \quad \text{ with }\alpha_2\beta_1-\alpha_1\beta_2> 0,
\end{equation}
we then choose the boundary condition for $\bar u_1$ and $\bar u_2$ according to \eqref{eq2.2.2} and Remark \ref{rmk2.2.1} depending on the signs of $\alpha_1$ and $\alpha_2$. Similarly, we choose the boundary conditions for all the other components of $\bar u$ according to the type of the pair $(C_i,D_i)$ ($i=1,\cdots,m$).

Therefore, combining Theorem \ref{thm2.1.1}, Remark \ref{rmk2.1.2}, Theorem \ref{thm2.2.2} and Remark \ref{rmk2.2.2}, we obtain the following result.
\begin{lemma}\label{lem2.4.1}
The operator $-\bar A$ is the infinitesimal generator of a contraction semigroup on $H^n=L^2(\Omega)^n$.
\end{lemma}

We are now able to define the unbounded operator $A$ on $H^n$ with $Au=P^{-t}\bar A P^{-1}u$, $\forall u\in\mathcal D(A)$, and
\[
\mathcal D(A)=\aiminset{u\in H^n=L^2(\Omega)^n\,:\,u=P\bar u,\,\bar u\in\mathcal D(\bar A)}.
\]
By virtue of Theorem \ref{T-thm2.5} and Lemma \ref{lem2.4.1}, we obtain our main result.
\begin{thm}\label{thm2.4.1}
The operator $-A$ is the infinitesimal generator of a contraction semigroup on $H^n=L^2(\Omega)^n$.
\end{thm}

Direct computations show that $Au=A_1u_x + A_2u_y$, and we thus obtain that the initial and boundary value problem \eqref{eq2.4.1} is equivalent to the abstract initial value problem
\begin{equation}\begin{cases}\label{eq2.4.4}
\frac{\text{d}u}{\text{d}t} + Au = f,\\
u(0)=u_0.
\end{cases}\end{equation}
The \emph{suitable boundary conditions} are already taken into account in the domain of $\mathcal D(A)$. Thanks to Theorem \ref{thm2.4.1} this problem is now solved by the Hille-Yoshida theorem and we have the following:
\begin{thm}\label{thm2.4.2}
Assume that $A_1,A_2$ are two non-singular real symmetric matrices, and that $A_1^{-1}A_2$ is diagonalizable over $\mathbb C$. 
Let $H=L^2(\Omega)$, and $(A,\mathcal{D}(A))$ be defined as before. Then the initial value problem \eqref{eq2.4.4} is well-posed. That is, 
\begin{enumerate}[i)]

\item for every $u_0\in H^n$, and $f\in L^1(0,T;H^n)$, the problem \eqref{eq2.4.4} admits a unique weak solution $u\in \mathcal{C}([0,T];H^n)$ satisfying
\begin{equation*}
u(t)=S(t)u_0 + \int_0^t S(t-s)f(s)ds,\,\forall t\in[0,T],
\end{equation*}
where $(S(t))_{t\geq 0}$ is the contraction semigroup generated by the operator $-A$;

\item for every $u_0\in\mathcal{D}(A)$, and $f\in L^1(0,T;H^n)$, with $f'=\text{d}f/\text{d}t\in L^1(0,T;H^n)$, the problem \eqref{eq2.4.4} has a unique strong solution $u$ such that
\begin{equation*}
u\in \mathcal{C}\big([0,T];\mathcal{D}(A)\big),\quad \frac{\text{d}u}{\text{d}t}\in \mathcal{C}\big([0,T];H^n\big).
\end{equation*}
\end{enumerate}

\end{thm}

\section{The variable coefficients case}\label{sec3}
In this section, we will show how our results in Section \ref{sec2} can be extended to the variable coefficients case. That is we want to study the IBVP \eqref{eq2.0.1} with variable coefficients. As in Section \ref{sec2}, we first study two fundamental hyperbolic and elliptic problems and then generalize to the full system by diagonalization.
We should bear in mind that all the functions in the section depend on the space variables $(x,y)$.
\subsection{The scalar (hyperbolic) case}\label{subsec3.1}
In this subsection, we consider the scalar equation ($n=1$ in \eqref{eq2.0.1}) with variable coefficients:
\begin{equation}\begin{cases}\label{eq3.1.1}
  u_t + a_1(x,y)u_x + a_2(x,y)u_y = f,\\
  u(0,x,y)=u_0(x,y),
\end{cases}\end{equation}
where $(x,y)\in\Omega$, $t\in(0,T)$, and $a_1(x,y),a_2(x,y)$ are either positive or negative away from zero everywhere on $\Omega$. Similarly as in Subsection \ref{subsec2.1}, we only consider the case when both $a_1$ and $a_2$ are positive away from zero. Hence, we assume that $a_1,a_2$ are $\mathcal C^1(\overline\Omega)$-functions that satisfy
\begin{equation}\label{asp3.1.1}
  a_1(x,y),\,a_2(x,y)\geq c_0,
\end{equation}
for some constant $c_0>0$.
We then associate to \eqref{eq3.1.1} the following boundary conditions
\begin{equation}\begin{split}\label{eq3.1.3}
  u=0,\quad\text{ on }\Gamma_W\cup\Gamma_S=\aiminset{x=0}\cup\aiminset{y=0}.\\
\end{split}\end{equation}
As in Subsection \ref{subsec2.1}, we define the unbounded operator $\mathcal T_1$ on $H=L^2(\Omega)$ with
$\mathcal T_1 u=a_1u_x + a_2u_y$, $\forall\,u\in\mathcal D(\mathcal T_1)$, and 
\[
\mathcal D(\mathcal T_1)=\aiminset{u\in H=L^2(\Omega)\,:\,\mathcal T_1u=a_1u_x + a_2u_y \in H,\text{ and }u\text{ satisfies }\eqref{eq3.1.3}}.
\]
If $u$ and $\mathcal T_1u$ both belong to $L^2(\Omega)$, then the traces of $u$ at the boundary are well-defined by using the same arguments as in Proposition \ref{M-propb.2} and the assumption \eqref{asp3.1.1}. Hence, the domain $\mathcal D(\mathcal T_1)$ is well-defined.

Applying Theorem \ref{T-thm3.1} with $\lambda(x,y) = a_1(x,y)/a_2(x,y)$, we obtain the following result: 
\begin{lemma}\label{lem3.1.1}
$\mathcal V(\Omega)\cap\mathcal D(\mathcal T_1)$ is dense in $\mathcal D(\mathcal T_1)$.
\end{lemma}
\subsubsection{Quasi-positivity of $\mathcal T_1$ and its adjoint $\mathcal T_1^*$}
The operator $\mathcal T_1$ being quasi-positive means that for some $\omega_0>0$, $\aimininner{ \mathcal T_1u}{u}_H\geq -\omega_0 \aiminnorm{u}_{L^2}^2$ holds for all $u\in\mathcal D(\mathcal T_1)$, and we prove it as follows.
For $u$ smooth in $\mathcal{D}(\mathcal T_1)$, integrating by parts and using the boundary conditions \eqref{eq3.1.3} yield
\begin{equation}\begin{split}\label{eq3.1.4}
\aimininner{ \mathcal T_1u}{u}_H &= \int_\Omega (a_1u_x+a_2u_y)u\,\text{d}x\text{d}y \\
&=\frac12\big(\int_\Omega(-\partial_xa_1-\partial_ya_2)u^2\,\text{d}x\text{d}y+\int_0^{L_2} a_1u^2\big|_{x=L_1}\,\text{d}y + \int_0^{L_1} a_2u^2\big|_{y=L_2}\,\text{d}x\big)\\
&\geq \frac12\int_\Omega(-\partial_xa_1-\partial_ya_2)u^2\,\text{d}x\text{d}y\\
&\geq -\omega_0 \aiminnorm{u}_{L^2}^2,
\end{split}\end{equation}
where $\omega_0>0$ depends only on the norms of $a_1$ and $a_2$ in $\mathcal C^1(\overline\Omega)$.
This is also valid for all $u$ in $\mathcal{D}(\mathcal T_1)$ thanks to Lemma \ref{lem3.1.1}. 

We now turn to the definition of the formal adjoint $\mathcal T_1^*$ of $\mathcal T_1$ and its domain $\mathcal{D}(\mathcal T_1^*)$.
For $u\in\mathcal{D}(\mathcal T_1)$ and $\bar u\in H$ smooth, integrating by parts and using the boundary conditions \eqref{eq3.1.3}, we find
\begin{equation}\label{eq3.1.5}
\begin{split}
\aimininner{\mathcal T_1 u}{\bar u} &= \int_\Omega (a_1u_x+a_2u_y)\bar u\, \text{d}x\text{d}y \\
&=\int_\Omega \big( -\partial_x(a_1\bar u) - \partial_y(a_2\bar u) \big)u\,\text{d}x\text{d}y \\
&\hspace{20pt} + \int_0^{L_2} (a_1u\bar u)\big|_{x=L_1}\text{d}y - \int_0^{L_1} (a_2u\bar u)\big|_{y=L_2} \text{d}x.
\end{split}
\end{equation}
Therefore, similarly as in Subsection \ref{subsec2.1}, we can conclude from \eqref{eq3.1.5} that
\begin{equation}\begin{split}\label{eq3.1.6}
 \mathcal{T}_1^*\bar u& =-\partial_x(a_1\bar u) - \partial_y(a_2\bar u)\\
& =-a_1\bar u_x - a_2\bar u_y -(\partial_x a_1+\partial_y a_2)\bar u ,
\end{split}\end{equation}
and in order to guarantee that $u\mapsto\aimininner{\mathcal T_1 u}{\bar u}$ is continuous on $\mathcal D(\mathcal T_1)$ for the norm of $H^2$, the following boundary conditions must be satisfied:
\begin{equation}\label{eq3.1.7}
\bar u=0, \text{ on } \Gamma_E\cup\Gamma_N = \{x=L_1\}\cup\{y=L_2\};
\end{equation}
hence finally, the domain $\mathcal D(\mathcal T_1^*)$ of $\mathcal T_1^*$::
\begin{equation*}
\mathcal{D}(\mathcal T_1^*) = \aiminset{ \bar u\in H=L^2(\Omega)\,:\,\mathcal T_1^*\bar u\in H, \text{ and } \bar u\text{ satisfies } \eqref{eq3.1.7}}.
\end{equation*}
We have an equivalent characterization of the domain $\mathcal D(\mathcal T_1^*)$ since $a_1,a_2$ belong to $\mathcal C^1(\overline\Omega)$ (see also \eqref{eq3.1.6}), that is
\begin{equation*}
\mathcal{D}(\mathcal T_1^*) = \aiminset{ \bar u\in H\,:\,a_1\bar u_x + a_2\bar u_y\in H, \text{ and } \bar u\text{ satisfies } \eqref{eq3.1.7}}.
\end{equation*}

Applying Theorem \ref{T-thm3.1} and Remark \ref{T-rmk3.2}, we once again have
\begin{lemma}\label{lem3.1.2}
$\mathcal{V}^*(\Omega)\cap\mathcal{D}(\mathcal T_1^*)$ is dense in $\mathcal{D}(\mathcal T_1^*)$.
\end{lemma}
We now prove that $\mathcal T_1^*$ is quasi-positive. For $\bar u\in\mathcal{D}(\mathcal T_1^*)$ smooth, integrating by parts and using the boundary conditions \eqref{eq3.1.7} yields
\begin{equation}\begin{split}
  \aimininner{ \mathcal T_1^*\bar u}{\bar u}_H &=\int_\Omega (-a_1(x,y)\bar u_x - a_2(x,y)\bar u_y -(\partial_x a_1+\partial_y a_2)\bar u)\bar u\,\text{d}x\text{d}y\\
  &=\frac12\int_\Omega (\partial_xa_1+\partial_ya_2)\bar u^2\,\text{d}x\text{d}y - \int_\Omega(\partial_x a_1+\partial_y a_2)\bar u^2\,\text{d}x\text{d}y\\
  &\hspace{20pt} +\frac12 \int_0^{L_2} (a_1u\bar u)\big|_{x=0}\text{d}y +\frac12\int_0^{L_1} (a_2u\bar u)\big|_{y=0} \text{d}x.\\
  &=-\frac12\int_\Omega(\partial_x a_1+\partial_y a_2)\bar u^2\,\text{d}x\text{d}y\\
  &\geq -\omega_0\aiminnorm{\bar u}^2_{L^2},
\end{split}\end{equation}
which is also valid for all $\bar u\in\mathcal D(\mathcal T_1^*)$ thanks to Lemma \ref{lem3.1.2}, where $\omega_0$ is the same one appearing in \eqref{eq3.1.4}.

\subsubsection{Semigroup} We are now ready to prove the main theorem in this subsection.
\begin{thm}\label{thm3.1.1}
The operator $-\mathcal T_1$ is the infinitesimal generator of a quasi-contraction semigroup on $H=L^2(\Omega)$.
\end{thm}
\begin{proof}
The proof relies on Theorem \ref{S-thm2.5} and the similar arguments as in the proof of Theorem \ref{thm2.1.1} for the closedness of the operators $\mathcal T_1$ and $\mathcal T_1^*$ and the density of their domains. We omit the details.
\end{proof}

\subsection{The simple (elliptic) system case}\label{subsec3.2}

In this subsection, we continue to consider the following simple system (the case when $n=2$ in \eqref{eq2.0.1})
\begin{equation}\begin{cases}\label{eq3.2.1}
  u_t+T_1u_x + T_2u_y = f,\\
  u(0)=u_0,
\end{cases}\end{equation}
where $u=(u_1,u_2)^t$, $f=(f_1,f_2)^t$, and 
\begin{equation}\label{eq3.2.1.1}
T_1=\begin{pmatrix}
\alpha_1 & \beta_1\\
\beta_1 & -\alpha_1
\end{pmatrix},\hspace{6pt}
T_2=\begin{pmatrix}
\alpha_2 & \beta_2\\
\beta_2 & -\alpha_2
\end{pmatrix}.
\end{equation}
Here, we assume that $\alpha_1,\alpha_2,\beta_1,\beta_2$ are in $\mathcal C^{1,\gamma}(\overline\Omega)$\footnote{We need the H\"older continuity in Theorem \ref{thma.2}} for some $0<\gamma<1$ satisfying (see the conclusion in Theorem \ref{thm2.3.1}):
 \begin{equation}\label{asp3.2.1}
 \alpha_2\beta_1-\alpha_1\beta_2 \equiv 1.
 \end{equation}
Note that under the assumption \eqref{asp3.2.1}, both $T_1$ and $T_2$ are non-singular.

Here, we only consider the case when $\alpha_1,\alpha_2$ are positive away from zero, and the other cases when $\alpha_1$ or $\alpha_2$ are negative from zero would be similar. We thus assume that
\begin{equation}\label{asp3.2.2}
\alpha_1,\,\alpha_2 \geq c_0,
\end{equation}
for some constant $c_0>0$, and choose the following homogeneous boundary conditions
\begin{equation}\label{eq3.2.2}
\begin{cases}
 u_1 = 0 \text{ on } \Gamma=\Gamma_W\cup\Gamma_S=\aiminset{x=0}\cup\aiminset{y=0}, \\
 u_2 = 0 \text{ on } \Gamma^c=\Gamma_E\cup\Gamma_N=\aiminset{x=L_1}\cup\aiminset{y=L_2}.
 \end{cases}
\end{equation}
Here $\Gamma^c$ is the complement of $\Gamma$ with respect to the boundary $\partial\Omega$. 

We now define the unbounded operator $\mathcal T_2$ on $H^2:=L^2(\Omega)^2$ by setting
\begin{equation}\label{eq3.2.3}
\mathcal T_2 u=T_1 u_x+T_2 u_y,\quad\forall\,u\in \mathcal D(\mathcal T_2),
\end{equation}
with
\[
 \mathcal D(\mathcal T_2)=\aiminset{  u=(u_1,u_2)^t\in H^2=L^2(\Omega)^2\,:\, \mathcal T_2 u\in H^2, u\text{ satisfies \eqref{eq3.2.2} } },
\]
and recall the function space $V$ defined in Subsection \ref{subsec2.2}:
\begin{equation*}
V=\{ u=(u_1,u_2)^t\in H^1(\Omega)^2 \ |\ u\text{ satisfies \eqref{eq3.2.2} } \}.
\end{equation*}

\subsubsection{Properties of $\mathcal T_2$}
We proceed similarly as in Subsection \ref{subsec2.2}.
\begin{prop}\label{prop3.2.1}
The domain $\mathcal D(\mathcal T_2)$ of $\mathcal T_2$ is the space $V$.
\end{prop}
The proof of Proposition \ref{prop3.2.1} is exactly the same as the proof for \cite[Theorem 4.1]{HT12} (Proposition \ref{prop2.2.1}) except that we need the assumption \eqref{asp3.2.1} to dispense the last term in the integrals of (4.9) in \cite{HT12} and utilize Theorem \ref{thma.1} instead of \cite[Proposition 4.1]{HT12}. 
In the variable coefficients case, the new additional difficulties (as compared to the constant coefficients) appear in the proof of Theorem \ref{thma.1} which is essentially based on the existence of solutions for the Beltrami equations and the use of quasi-conformal mappings.

\begin{prop}\label{prop3.2.2}
The operator $\mathcal T_2$ is quasi-positive in the sense that for any $ u\in \mathcal D(\mathcal T_2)= V$:
\begin{equation}\label{eq3.2.5}
  \aimininner{\mathcal T_2  u}{ u}\geq -\omega_0\aiminnorm{ u}^2_{L^2},
\end{equation}
where $\omega_0$ is a positive constant, only depending on the norms of $\alpha_1,\alpha_2,\beta_1,\beta_2$ in $\mathcal C^1(\overline\Omega)$.
\end{prop}
\begin{proof}
We prove \eqref{eq3.2.5} by direct computation. For $ u\in V$, integrations by parts yield
\begin{equation*}
\begin{split}
\aimininner{\mathcal T_2  u}{ u}&=\int_\Omega   u^t T_1 u_x +   u^t T_2 u_y \text{d}x\text{d}y\\
&= \frac12\int_\Omega - u^t T_{1x} u -   u^t T_{2y} u \,\text{d}x\text{d}y +\frac12\int_0^{L_2} u^t T_1 u\Big|_{x=0}^{x=L_1}\, \text{d}y + \frac12\int_0^{L_1} u^t T_2 u\Big|_{y=0}^{y=L_2}\, \text{d}x\\
&=(\text{using the boundary conditions \eqref{eq3.2.2} and the assumption $\alpha_1,\alpha_2>0$}) \\
&\geq  \frac12\int_\Omega - u^t T_{1x} u -   u^t T_{2y} u \,\text{d}x\text{d}y \\
&\geq -\omega_0\aiminnorm{ u}_{L^2}^2,
\end{split}
\end{equation*}
where $\omega_0>0$ only depends on the norms of $\alpha_1,\alpha_2,\beta_1,\beta_2$ in $\mathcal C^1(\overline\Omega)$.
\end{proof}

\subsubsection{The adjoint operator $\mathcal T_2^*$} In the variable coefficients case, we can not prove directly a similar result as in Theorem \ref{thm2.2.1}, and in order to use the semigroup theory, we turn to the adjoint operator $\mathcal T_2^*$ of $\mathcal T_2$. For $u\in \mathcal D(\Omega)\subset \mathcal D(\mathcal T_2)$ and $\bar u$ smooth, integrations by parts yield
\begin{equation}\begin{split}\label{eq3.2.7}
\aimininner{\mathcal T_2 u}{\bar u} = \int_\Omega \bar u^t\cdot (T_1 u_x + T_2 u_y)\,\text{d}x\text{d}y
&=\int_\Omega u^t\cdot[ -(T_1\bar u)_x - (T_2\bar u)_y]\,\text{d}x\text{d}y.
\end{split}\end{equation}
Hence, in order to guarantee that $u \mapsto \aimininner{\mathcal T_2 u}{\bar u}$ is continuous on $\mathcal D(\mathcal T_2)$ for the norm of $L^2(\Omega)^2$ (see Subsection \ref{subsec2.1}), $-(T_1\bar u)_x - (T_2\bar u)_y$ must be in $L^2(\Omega)^2$.
Using the notation $\mathcal X(\Omega)$ introduced in Appendix \ref{sec-integration} with $T_1$ and $T_2$ defined by \eqref{eq3.2.1.1}, we find that $\bar u$ belongs to $\mathcal X(\Omega)$.

Now for $u\in\mathcal D(\mathcal T_2) = V$ and $\bar u \in \mathcal X(\Omega)$, by Theorem \ref{thmg.1}, we obtain 
\begin{equation}\label{eq3.2.8}
\aimininner{\mathcal T_2 u}{\bar u} = \aimininner{T_1u_x + T_2u_y}{\bar u} = \aimininner{u}{-(T_1\bar u)_x - (T_2\bar u)_y} + \aimininner{\gamma_\nu \bar u}{\gamma_0 u},
\end{equation}
where, specifically, 
\begin{equation*}
\gamma_\nu \bar u=\begin{cases}
-T_1\bar u, \text{ on }\Gamma_W,\\
T_1\bar u,\text{ on }\Gamma_E,\\
-T_2\bar u,\text{ on }\Gamma_S,\\
T_2\bar u,\text{ on }\Gamma_N.
\end{cases}
\end{equation*}
Therefore,  in order to guarantee that $u \mapsto \aimininner{\mathcal T_2 u}{\bar u}$ is continuous on $\mathcal D(\mathcal T_2)$ for the norm of $L^2(\Omega)^2$, we must have
\begin{equation}\label{eq3.2.9}
\aimininner{\gamma_\nu \bar u}{\gamma_0 u} = 0.
\end{equation}
Since $u\in\mathcal D(\mathcal T_2)$ satisfies the boundary conditions \eqref{eq3.2.2}, we infer from \eqref{eq3.2.9} that $\bar u$ must satisfy the following boundary conditions
\begin{equation}\begin{cases}\label{eq3.2.10}
\beta_1\bar u_1 - \alpha_1\bar u_2 = 0,\text{ on }\Gamma_W=\aiminset{x=0},\\
\alpha_1\bar u_1 + \beta_1\bar u_2 = 0,\text{ on }\Gamma_E=\aiminset{x=L_1},\\
\beta_2\bar u_1 - \alpha_2\bar u_2 = 0,\text{ on }\Gamma_S=\aiminset{y=0},\\
\alpha_2\bar u_1 + \beta_2\bar u_2 = 0,\text{ on }\Gamma_N=\aiminset{y=L_2}.
\end{cases}\end{equation}

Therefore, we can conclude that the domain $\mathcal D(\mathcal T_2^*)$ of the adjoint operator $\mathcal T_2^*$ satisfies 
\[
\mathcal D(\mathcal T_2^*)\subset \aiminset{ \bar u=(\bar u_1,\bar u_2)^t\in L^2(\Omega)^2\,|\, \mathcal T_2^*\bar u \in L^2(\Omega)^2,\; \bar u\text{ satisfies }\eqref{eq3.2.10}} = :\widetilde{\mathcal D}(\mathcal T_2^*),
\]
where $\mathcal T_2^*\bar u= -(T_1\bar u)_x - (T_2\bar u)_y$. Now we have for any $u\in \mathcal D(\mathcal T_2)$ and any $\bar u\in   \widetilde{\mathcal D}(\mathcal T_2^*)$, we have by \eqref{eq3.2.8}
\[
\aimininner{\mathcal T_2 u}{\bar u} =  \aimininner{u}{ \mathcal T_2^* \bar u},
\]
which implies that $\widetilde{\mathcal D}(\mathcal T_2^*)\subset \mathcal D(\mathcal T_2^*)$. Hence
\[
\mathcal D(\mathcal T_2^*) = \aiminset{ \bar u=(\bar u_1,\bar u_2)^t\in L^2(\Omega)^2\,|\, \mathcal T_2^*\bar u \in L^2(\Omega)^2,\; \bar u\text{ satisfies }\eqref{eq3.2.10}}.
\]

We now use the duality method as in \cite{LM72} to prove the following density result.
\begin{prop}\label{prop3.2.3}
$\mathcal C^\infty(\overline\Omega)\cap \mathcal D(\mathcal T_2^*)$ is dense in $\mathcal D(\mathcal T_2^*)$.
\end{prop}
\begin{proof}
Let $M\in \big(\mathcal D(\mathcal T_2^*)\big)'$, the dual space of $\mathcal D(\mathcal T_2^*)$, and assume that the restriction of $M$ on $\mathcal C^\infty(\overline\Omega)\cap \mathcal D(\mathcal T_2^*)$ is zero, i.e.
\[
M(\bar u) = 0,\quad\quad \forall\,\bar u\in \mathcal C^\infty(\overline\Omega)\cap \mathcal D(\mathcal T_2^*).
\]
We are going to show that
\begin{equation}\label{eq3.2.11}
M(\bar u) = 0,\quad\quad \forall\,\bar u\in  \mathcal D(\mathcal T_2^*),
\end{equation}
and we can conclude from \eqref{eq3.2.11} that $\mathcal C^\infty(\overline\Omega)\cap \mathcal D(\mathcal T_2^*)$ is dense in $\mathcal D(\mathcal T_2^*)$.

For $M\in \big(\mathcal D(\mathcal T_2^*)\big)'$, observing that $\mathcal D(\mathcal T_2^*)\subset L^2(\Omega)^2\times L^2(\Omega)^2$ and using the Hahn-Banach Theorem, we see that there exists $(g,h)\in L^2(\Omega)^2\times L^2(\Omega)^2$ such that
\begin{equation}\label{eq3.2.11.1}
M(\bar u) = \aimininner{h}{\bar u} + \aimininner{g}{\mathcal T_2^* \bar u} =  \aimininner{h}{\bar u} + \aimininner{g}{ -(T_1\bar u)_x - (T_2\bar u)_y }.
\end{equation}

For $\bar u\in \mathcal D(\Omega)$, we have
\begin{equation*}\begin{split}
0=M(\bar u) &= \aimininner{h}{\bar u} + \aimininner{g}{ -(T_1\bar u)_x - (T_2\bar u)_y } \\
&=(\text{in the sense of distributions since }\bar u\in \mathcal D(\Omega)) \\
&= \aimininner{h}{\bar u} + \aimininner{g_x}{T_1\bar u} + \aimininner{g_y}{T_2\bar u}\\
&=(\text{since $T_1$ and $T_2$ are symmetric}) \\
&= \aimininner{h}{\bar u} + \aimininner{T_1g_x + T_2g_y}{\bar u},
\end{split}\end{equation*}
which shows that
\begin{equation}\label{eq3.2.12}
T_1g_x + T_2g_y  = -h \in L^2(\Omega)^2.
\end{equation}
Hence, $g$ belongs to $\mathcal X(\Omega)$ defined in Appendix \ref{sec-integration}, i.e.
\[
\mathcal X(\Omega)=\{ g\in L^2(\Omega)^2\,:\, T_1g_x + T_2g_y\in L^2(\Omega)^2\}.
\]
For $\bar u \in \mathcal C^\infty(\overline\Omega)\cap \mathcal D(\mathcal T_2^*)$, using Remark \ref{rmkg.1} which justifies the integration by parts, we find that
\begin{equation*}\begin{split}
0=M(\bar u)&=\aimininner{h}{\bar u} + \aimininner{g}{ -(T_1\bar u)_x - (T_2\bar u)_y } \\
&=\aimininner{h}{\bar u} +\aimininner{T_1g_x + T_2g_y}{\bar u} -\aimininner{\tilde\gamma_\nu g}{T_1\gamma_0\bar u + T_2\gamma_0\bar u},
\end{split}\end{equation*}
which, together with \eqref{eq3.2.12}, implies that 
\begin{equation}\label{eq3.2.13}
\aimininner{\tilde\gamma_\nu g}{T_1\gamma_0\bar u + T_2\gamma_0\bar u} = 0.
\end{equation}
Since $\bar u$ satisfies the boundary conditions \eqref{eq3.2.10}, we infer from \eqref{eq3.2.13} that $g$ satisfies the boundary conditions \eqref{eq3.2.2}. Therefore, we obtain that
\[
g\in \mathcal D(\mathcal T_2) = V.
\]

For any $\bar u\in\mathcal D(\mathcal T_2^*)$, using Theorem \ref{thmg.1} again and \eqref{eq3.2.9} with $u=g$, we find that
\[
\aimininner{g}{\mathcal T_2^*\bar u} = \aimininner{ T_1g_x + T_2g_y}{\bar u}.
\]
Therefore, for any $\bar u\in\mathcal D(\mathcal T_2^*)$, we have by \eqref{eq3.2.11.1} and \eqref{eq3.2.12}:
\[
M(\bar u) = \aimininner{h}{\bar u} + \aimininner{g}{\mathcal T_2^* \bar u} = \aimininner{h}{\bar u} + \aimininner{ T_1g_x + T_2g_y}{\bar u} = \aimininner{ h+ T_1g_x + T_2g_y}{\bar u} = 0.
\]
We thus proved \eqref{eq3.2.11} and the result follows.
\end{proof}

\begin{prop}\label{prop3.2.4}
The operator $\mathcal T_2^*$ is quasi-positive in the sense that for any $\bar u\in \mathcal D(\mathcal T_2^*)$:
\begin{equation}\label{eq3.2.14}
  \aimininner{\mathcal T_2^* \bar u}{\bar u}\geq -\omega_0\aiminnorm{ \bar u}^2_{L^2},
\end{equation}
where $\omega_0$ is the same as in Proposition \ref{prop3.2.2}.
\end{prop}
\begin{proof}
First, for $ u\in \mathcal C^\infty(\overline\Omega)\cap \mathcal D(\mathcal T_2^*)$, integrations by parts yield
\begin{equation*}
\begin{split}
\aimininner{\mathcal T_2^* \bar u}{\bar u}&=-\int_\Omega  \bar u^t (T_1 \bar u)_x +  \bar u^t (T_2 \bar u)_y \text{d}x\text{d}y\\
&= \frac12\int_\Omega -\bar u^t T_{1x} \bar u -  \bar u^t T_{2y}\bar u \,\text{d}x\text{d}y -\frac12\int_0^{L_2}\bar u^t T_1\bar u\Big|_{x=0}^{x=L_1}\, \text{d}y - \frac12\int_0^{L_1} \bar u^t T_2 \bar u\Big|_{y=0}^{y=L_2}\, \text{d}x\\
&=(\text{using the boundary conditions \eqref{eq3.2.10} and the assumption $\alpha_1,\alpha_2>0$}) \\
&\geq  \frac12\int_\Omega - \bar u^t T_{1x}\bar u - \bar  u^t T_{2y} \bar u \,\text{d}x\text{d}y \\
&\geq -\omega_0\aiminnorm{\bar u}_{L^2}^2.
\end{split}
\end{equation*}
We can then conclude \eqref{eq3.2.14} by the density result Proposition \ref{prop3.2.3}.
\end{proof}

\subsubsection{Semigroup}
We are now ready to prove the main theorem in this subsection.
\begin{thm}\label{thm3.2.2}
The operator $-\mathcal T_2$ is the infinitesimal generator of a quasi-contraction semigroup on $H^2=L^2(\Omega)^2$.
\end{thm}
\begin{proof}
Since $\mathcal T_2$ is continuous from $V$ to $H^2$ and $V$ is dense in $H^2$, it is clear that $\mathcal T_2$ is a closed, densely defined operator on $H^2$. Combining Propositions \ref{prop3.2.2}, \ref{prop3.2.4} and applying Theorem \ref{S-thm2.5} to $\mathcal T_2$, we obtain the result.
\end{proof}

\subsection{The full system}\label{subsec3.3}
We are now ready to consider the full system \eqref{eq2.0.1} in the variable coefficients case, that is
\begin{equation}\begin{cases}\label{eq3.3.1}
u_t + A_1 u_x + A_2 u_y = f,\\
u(0)=u_0,\\
u\text{ satisfies \emph{suitable boundary conditions}},
\end{cases}\end{equation}
where $u=(u_1,\cdots,u_n)^t$, $f=(f_1,\cdots,f_n)^t$, and $A_1=A_1(x,y)$, $A_2=A_2(x,y)$ are real non-singular symmetric $n\times n$ matrices. 

\textbf{Main Assumptions:}
We assume that 
\addtocounter{equation}{1}
\newcounter{aiminspecialeqn}
\numberwithin{aiminspecialeqn}{section}
\setcounter{aiminspecialeqn}{\value{equation}}
\begin{enumerate}[\quad(\theaiminspecialeqn a) $\bullet$]
  \item $A_1,A_2$ belong to $\mathcal C^{1,\gamma}(\overline\Omega)$ for some $0<\gamma<1$,\label{eq3.3.2a}
\end{enumerate}
and as $(x,y)$ vary in $\overline\Omega$:
\begin{enumerate}[\quad(\theaiminspecialeqn a) $\bullet$]
  \setcounter{enumi}{1}
  \item all eigenvalues of $A_1,A_2$ remain either positive away or negative away from zero,\label{eq3.3.2b}
  \item all the real eigenvalues of $A_1^{-1}A_2$ and all the imaginary part of complex eigenvalues of $A_1^{-1}A_2$ remain either positive away or negative away from zero,\label{eq3.3.2c}
  \item the multiplicities of the eigenvalues of $A_1^{-1}A_2$ remain individually constant.\label{eq3.3.2d}
\end{enumerate}

By Theorem \ref{thm2.3.1}, we know that there exists a non-singular matrix $P=P(x,y)$ which can diagonalize $A_1$ and $A_2$ simultaneously, i.e.
\begin{equation}
  \begin{split}
P^tA_1P=\bar A_1=\text{diag}(C_1,\cdots,C_m),\\
P^tA_2P=\bar A_2=\text{diag}(D_1,\cdots,D_m),\\  
  \end{split}
\end{equation}
for some integer $m$ satisfying $1\leq m\leq n$, where the pair $(C_i,D_i)$ ($i=1,\cdots,m$) satisfies the conclusion in Theorem \ref{thm2.3.1}. As in Subsection \ref{subsec2.4}, we first define the unbounded operator $\bar A$ on $H^n=L^2(\Omega)^n$ with $\bar A\bar u=\bar A_1\bar u_x+\bar A_2\bar u_y$, $\forall\,\bar u\in \mathcal D(\bar A)$, and
\begin{equation}\label{eq3.3.3}
\mathcal D(\bar A)=\aiminset{\bar u\in H^n\,:\,\bar A\bar u\in H^n,\, \bar u\text{ satisfies \emph{suitable boundary conditions}}},
\end{equation}
where the \emph{suitable boundary conditions} were already explained in Subsection \ref{subsec2.4}.
Therefore, combining Theorem \ref{thm3.1.1} and Theorem \ref{thm3.2.2}, we obtain the following result.
\begin{lemma}\label{lem3.3.1}
The operator $-\bar A$ is the infinitesimal generator of a quasi-contraction semigroup on $H^n=L^2(\Omega)^n$.
\end{lemma}

We then define the unbounded operator $B_0$ on $H^n$ with $B_0u=P^{-t}\bar A P^{-1}u$, $\forall u\in\mathcal D(B_0)$, and
\[
\mathcal D(B_0)=\aiminset{u\in H^n=L^2(\Omega)^n\,:\,u=P\bar u,\,\bar u\in\mathcal D(\bar A)}.
\]
By virtue of Theorem \ref{T-thm2.5} and Lemma \ref{lem3.3.1}, we find that
the operator $-B_0$ is the infinitesimal generator of a quasi-contraction semigroup on $H^n=L^2(\Omega)^n$.
Direct computations show that $B_0u=A_1P(P^{-1}u)_x + A_2P(P^{-1}u)_y$. Set $B_1u=A_1P_xP^{-1}u+A_2P_yP^{-1}u$, and we see that $B_1$ is a linear bounded operator on $H^n$ by assumption (\theaiminspecialeqn\ref{eq3.3.2a}). We now define the unbounded operator $A$ on $H^n$ by setting $Au=B_0u+B_1u$ with $\mathcal D(A)=\mathcal D(B_0)$. According to the Bounded Perturbation Theorem \ref{S-thm2.3}, we obtain our main result.
\begin{thm}\label{thm3.3.1}
We assume that (\theaiminspecialeqn\ref{eq3.3.2a})-(\theaiminspecialeqn\ref{eq3.3.2d}) hold. Then the operator $-A$ is the infinitesimal generator of a quasi-contraction semigroup on $H^n=L^2(\Omega)^n$.
\end{thm}

Direct computations show that $Au=A_1u_x + A_2u_y$, and we thus obtain that the initial and boundary value problem \eqref{eq3.3.1} is equivalent to the abstract initial value problem
\begin{equation}\begin{cases}\label{eq3.3.4}
\frac{\text{d}u}{\text{d}t} + Au = f,\\
u(0)=u_0.
\end{cases}\end{equation}
The \emph{suitable boundary conditions} are already taken into account in the domain of $\mathcal D(A)$. Thanks to Theorem \ref{thm3.3.1} this problem is now solved by the Hille-Yoshida theorem exactly as for Theorem \ref{thm2.4.2}; we omit the details here.

\section{Applications}\label{sec4}
In this section, we consider some applications of our general results on hyperbolic partial differential equations in a rectangle. We mainly focus on showing that the equations in consideration satisfy the assumptions in Theorem~\ref{thm2.4.2}, or can be transformed into the desired diagonalization form (see Theorem~\ref{thm2.3.1}).

Here, we only consider the constant coefficients case for the sake of convenience, the variable coefficients case would be similar with suitable assumptions.
\subsection{The inviscid shallow water equations}\label{subsec4.1}
The linearized 2d inviscid shallow water equations (SWEs) have been directly studied in \cite{HT12} in the case of constant coefficients case; let us briefly show that those equations satisfy the assumptions in Theorem~\ref{thm2.4.2}.
The linearized 2d inviscid SWEs read
\begin{equation}\label{eq4.1.1}
\begin{cases}
u_t+u_0u_x + v_0u_y + g\phi_x -fv= 0, \\
v_t+u_0v_x + v_0v_y + g\phi_y + fu= 0, \\
\phi_t+u_0\phi_x + v_0\phi_y + \phi_0(u_x+v_y) = 0;
\end{cases}
\end{equation}
where $u,v$ are the horizontal velocities, $\phi$ is the height of the fluid in consideration, and $u_0,v_0$ are the reference velocities, $\phi_0$ is the reference height, and $g$ is the gravitational acceleration, $f$ is the Coriolis parameter. We assume that $u_0,v_0,\phi_0$ are positive constants and only consider the generic case (see \cite[Section 1]{HT12}) where:
\[
u_0^2\ne g\phi_0, \quad v_0^2\ne g\phi_0,\quad u_0^2+v_0^2\neq g\phi_0.
\]
 Equations \eqref{eq4.1.1} are a set of hyperbolic partial differential equations, and can be written in the compact form
\begin{equation}\label{eq4.1.2}
U_t + \mathcal{E}_1U_x + \mathcal{E}_2U_y + BU = 0,
\end{equation}
where $U=(u,v,\phi)^t$, $BU=(-fv,fu,0)^t$ and
\begin{equation*}
\mathcal{E}_1=\begin{pmatrix}
u_0&0&g\\
0&u_0&0\\
\phi_0&0&u_0
\end{pmatrix},\hspace{6pt}
\mathcal{E}_2=\begin{pmatrix}
v_0&0&0\\
0&v_0&g\\
0&\phi_0&v_0
\end{pmatrix}.
\end{equation*}
We observe that \eqref{eq4.1.2} is \emph{Friedrichs symmetrizable}, i.e. $\mathcal{E}_1$, $\mathcal{E}_2$ admit a symmetrizer $S_0=\text{diag}(1,1,g/\phi_0)$. As indicated at the beginning of Section \ref{sec2}, we make the change of variable $\widetilde U=S_0^{1/2}U$, and rewrite \eqref{eq4.1.2} in the symmetric form: 
\begin{equation}\label{eq4.1.3}
\widetilde U_t + \widetilde{\mathcal{E}}_1\widetilde U_x + \widetilde{\mathcal{E}}_2\widetilde U_y + \widetilde B\widetilde U = 0,
\end{equation}
where $\widetilde B = S_0^{1/2}BS_0^{-1/2}$, and 
\begin{equation*}
\widetilde{\mathcal{E}}_1=S_0^{1/2}\mathcal{E}_1S_0^{-1/2}=\begin{pmatrix}
u_0&0&\sqrt{g\phi_0}\\
0& u_0&0\\
\sqrt{g\phi_0}&0& u_0
\end{pmatrix},\quad
\widetilde{\mathcal{E}}_2=S_0^{1/2}\mathcal{E}_2S_0^{-1/2}=\begin{pmatrix}
v_0&0&0\\
0& v_0&\sqrt{g\phi_0}\\
0&\sqrt{g\phi_0}& v_0
\end{pmatrix}.
\end{equation*}
According to Theorem \ref{thm2.4.2}, we only need to show that $ \widetilde{\mathcal{E}}_1^{-1} \widetilde{\mathcal{E}}_2$ is diagonalizable over $\mathbb C$. Note that $\widetilde{\mathcal{E}}_1^{-1} \widetilde{\mathcal{E}}_2=S_0^{1/2}\mathcal{E}_1^{-1} \mathcal{E}_2S_0^{-1/2}$, and direct computation yields
\begin{equation}
P^{-1}\cdotp \mathcal{E}_1^{-1}\mathcal{E}_2\cdotp P = diag(\lambda_1, \lambda_2, \lambda_3),
\end{equation}
where $P$ has a complicated expression, whereas
\begin{equation*}
P^{-1}=\begin{pmatrix}
&\frac{v_0}{2\kappa_0} &-\frac{u_0}{2\kappa_0} &\frac{1}{2} \\
&-\frac{v_0}{2\kappa_0} &\frac{u_0}{2\kappa_0} &\frac{1}{2} \\
&\frac{u_0v_0}{u_0^2+v_0^2}&\frac{v_0^2}{u_0^2+v_0^2}&\frac{gv_0}{u_0^2+v_0^2}
\end{pmatrix},
\end{equation*}
where $\kappa_0=\sqrt{g(u_0^2+v_0^2-g\phi_0)/\phi_0}$, and
\begin{equation}
\lambda_1=\frac{u_0v_0+\phi_0\kappa_0}{u_0^2-g\phi_0},\hspace{6pt}
\lambda_2=\frac{u_0v_0-\phi_0\kappa_0}{u_0^2-g\phi_0},\hspace{6pt}
\lambda_3=\frac{v_0}{u_0}.
\end{equation}
Therefore, we can conclude that $\widetilde{\mathcal{E}}_1^{-1} \widetilde{\mathcal{E}}_2$ is diagonalizable over $\mathbb C$.

We finally remark that in \cite{HT12} we only studied the linearized SWEs in the constant coefficients case, but now with the results in Subsection \ref{subsec3.3}, we are also able to study the linearized SWEs in the variable coefficients case.

\subsection{The shallow water magnetohydrodynamics}\label{subsec4.3}
Recently, the equations of "shallow water" magnetohydrodynamics (SWMHD͒) have been proposed by Gilman in \cite{Gil00} in order to study the global dynamics of the solar tachocline, which is a thin layer in the solar interior at the base of the solar convection zone. In \cite{Ste01}, the author studied the properties of the SWMHD equations as a nonlinear system of hyperbolic conservation laws; and in \cite{Del02}, the author studies the Hamiltonian and symmetric hyperbolic structures of the SWMHD equations. 
With our general results on hyperbolic systems at hand, we are able to study the linearized 2d SWMHD equations in a rectangle. The linearized 2d SWMHD equations (see \cite{Gil00,Ste01,Del02}) read in compact form
\begin{equation}\label{eq4.3.1}
  U_t + \mathcal{E}_1 U_x + \mathcal{E}_2 U_y = 0,
\end{equation}
where $U=(u,v,b_1,b_2,\phi)^t$, $u,v$ are the fluid velocities, $b_1,b_2$ are the components of the magnetic field, $\phi$ is the height of the fluid in consideration, and 
\begin{equation}
\mathcal{E}_1 = \begin{pmatrix}
u_0&0&-b_{10}&0&g\\
0&u_0&0&-b_{10}&0\\
-b_{10}&0&u_0&0&0\\
0&-b_{10}&0&u_0&0\\
\phi_0&0&0&0&u_0
\end{pmatrix},\quad
\mathcal{E}_2 = \begin{pmatrix}
v_0&0&-b_{20}&0&0\\
0&v_0&0&-b_{20}&g\\
-b_{20}&0&v_0&0&0\\
0&-b_{20}&0&v_0&0\\
0&\phi_0&0&0&v_0
\end{pmatrix}.
\end{equation}
Here, $(u_0,v_0)$ are the reference velocities, $(b_{10},b_{20})$ are the reference magnetic field, $\phi_0$ is the reference height, and $g$ is the gravitational acceleration. 

We observe that \eqref{eq4.3.1} is also \emph{Friedrichs symmetrizable}, i.e. $\mathcal{E}_1$, $\mathcal{E}_2$ admit a symmetrizer $S_0=\text{diag}(1,1,1,1,g/\phi_0)$. Following the same calculations as in Subsection \ref{subsec4.1}, the initial and boundary value problem for \eqref{eq4.3.1} is well-posed as long as $\mathcal{E}_1^{-1}\mathcal{E}_2$ is diagonalizable over $\mathbb C$. We only consider the generic case in which we assume that
\begin{equation}\begin{cases}
u_0,v_0\neq 0,\quad u_0-b_{10},v_0-b_{20}\neq 0, \quad b_{10}^2 -u_0^2 + g\phi_0, b_{20}^2 -v_0^2 + g\phi_0\neq 0,\\
 (b_{10}b_{20} - u_0v_0)^2 - (b_{10}^2 -u_0^2 + g\phi_0)(b_{20}^2 -v_0^2 + g\phi_0)\neq 0.
\end{cases}\end{equation}
Direct computations shows that $\mathcal{E}_1^{-1}\mathcal{E}_2$ have five different eigenvalues which are (see \cite{Ste01})
\begin{equation}\begin{split}
\lambda_{1,2} &= \frac{b_{20} \pm v_0}{ b_{10} \pm u_0}, \quad  \lambda_5=\frac{v_0}{u_0}, \\
\lambda_{3,4} &= \frac{ b_{10}b_{20} - u_0v_0}{ b_{10}^2 -u_0^2 + g\phi_0 } \pm \frac{\sqrt{ (b_{10}b_{20} - u_0v_0)^2 - (b_{10}^2 -u_0^2 + g\phi_0)(b_{20}^2 -v_0^2 + g\phi_0) }}{ b_{10}^2 -u_0^2 + g\phi_0}.
\end{split}\end{equation}
Therefore, $\mathcal{E}_1^{-1}\mathcal{E}_2$ is diagonalizable over $\mathbb C$ and thus \eqref{eq4.3.1} is well-posed under suitable initial and boundary conditions.

\subsection{The Euler equation}\label{subsec4.4}
The motion of a compressible, inviscid fluid in the absence of heat convection is governed by the Euler equations, consisting of the mass, momentum and energy conservation laws (see e.g. \cite[Chapter 8]{Lio98}):
\begin{equation}\begin{cases}\label{eq4.4.1}
\partial_t \rho + \boldsymbol u\cdot\nabla\rho + \rho\nabla\cdot\boldsymbol u = 0,\\
\partial_t \boldsymbol u + (\boldsymbol u\cdot\nabla)\boldsymbol u + \rho^{-1}\nabla p=0,\\
\partial_t e + \boldsymbol u\cdot\nabla e + \rho^{-1}p\nabla\cdot\boldsymbol u = 0.
\end{cases}\end{equation}
where $\rho$ is the density, $\boldsymbol u$ is the velocity, $e$ is the internal energy and $p$ is the pressure. The equation of state (pressure law) reads
\begin{equation}\label{eq4.4.2}
  p=p(\rho,e).
\end{equation}
Here, we consider the linearized two dimensional Euler equations in a rectangle. Hence $\boldsymbol u=(u,v)^t$, and the Euler equations \eqref{eq4.4.1} linearized around the reference state $(u_0,v_0,\rho_0,e_0)$ becomes
\begin{equation}\label{eq4.4.3}
  U_t + \mathcal E_1 U_x + \mathcal E_2 U_y = 0,
\end{equation}
where $U=(u,v,\rho,e)^t$, and the two matrices $\mathcal E_1$ and $\mathcal E_2$ are
\begin{equation*}
  \mathcal E_1 = \begin{pmatrix}
    u_0&0&\frac{1}{\rho_0}\frac{\partial p}{\partial \rho}(\rho_0,e_0) & \frac{1}{\rho_0}\frac{\partial p}{\partial e}(\rho_0,e_0)\\
    0 & u_0 & 0 & 0 \\
    \rho_0 & 0 & u_0 & 0\\
    \frac{1}{\rho_0} p_0 & 0 & 0 & u_0
  \end{pmatrix},\; \mathcal E_2 = \begin{pmatrix}
    v_0 & 0 & 0 & 0\\
    0 & v_0 & \frac{1}{\rho_0}\frac{\partial p}{\partial \rho}(\rho_0,e_0) & \frac{1}{\rho_0}\frac{\partial p}{\partial e}(\rho_0,e_0)\\
    0 & \rho_0 & v_0 & 0 \\
    0 & \frac{1}{\rho_0} p_0 & 0 & v_0
  \end{pmatrix},
\end{equation*}
where $p_0 = p(\rho_0,e_0)$. We set 
$$S_0=\text{diag}(1,\,1,\,\frac{1}{\rho_0^2}\frac{\partial p}{\partial \rho}(\rho_0,e_0),\, \frac{1}{p(\rho_0,e_0)}\frac{\partial p}{\partial e}(\rho_0,e_0)),$$
and find that $S_0\mathcal E_1$ and $S_0\mathcal E_2$ are both symmetric. Therefore, with suitable assumptions on the reference state $(u_0,v_0,\rho_0,e_0)$ and the pressure law \eqref{eq4.4.2}, and following similar calculations as in Subsections \ref{subsec4.1}-\ref{subsec4.3}, we can obtain that the initial and boundary value problem \eqref{eq4.4.3} in a rectangle is well-posed.

\subsection{The wave equation}\label{subsec4.2}
The 2d wave equation in the first quadrant of the plane has already been studied in \cite{Tan78} and the multi-dimensional wave equation in a multi-dimensional corner domain has been studied in \cite{KO71}. Here, we consider the 2d wave equation in a rectangle:
\begin{equation}\label{eq4.2.1}
\frac{\partial^2 v}{\partial t^2}-\frac{\partial^2 v}{\partial x^2}-\frac{\partial^2 v}{\partial y^2} = h.
\end{equation}
We first reduce \eqref{eq4.2.1} to a symmetric hyperbolic system, for that purpose, we set
\begin{equation}\begin{cases}
u_1 = v_y - \alpha v_t,\\
u_2 = v_x - \beta v_t,
\end{cases}\end{equation}
where $\alpha^2 + \beta^2 = 1$, $\alpha$ and $\beta$ are real constants, $u=(u_1,u_2)^t$. By direct calculation, \eqref{eq4.2.1} is transformed into the system (see also \cite[Lemma 3.2]{Tan78}):
\begin{equation}\label{eq4.2.3}
\begin{pmatrix}
u_1\\
u_2
\end{pmatrix}_t + \begin{pmatrix}
-\beta&\alpha\\
\alpha&\beta
\end{pmatrix}\begin{pmatrix}
u_1\\
u_2
\end{pmatrix}_x + 
\begin{pmatrix}
\alpha&\beta\\
\beta&-\alpha
\end{pmatrix}\begin{pmatrix}
u_1\\
u_2
\end{pmatrix}_y = \begin{pmatrix}
-\alpha h\\
-\beta h
\end{pmatrix}.
\end{equation}
Then according to the results in Subsections \ref{subsec2.2}, we can obtain well-posedness results for system \eqref{eq4.2.3} under suitable boundary conditions, and we have infinitely many choice of boundary conditions. The case when $\alpha$ and $\beta$ are functions depending on the space variable $(x,y)$ can also be treated (see Subsection \ref{subsec3.2}). 

The differences between earlier works \cite{Tan78,KO71} and our work for the wave equation are the idea to assign the boundary conditions, the boundary conditions itself and the method implemented, where the Fourier transform method is used in both the articles \cite{Tan78,KO71}. 

\appendix

\section{Simultaneous diagonalization by congruence}\label{sec-diagonal}
In this appendix, we will prove a simultaneous diagonalization result by congruence which is essential for studying the full hyperbolic system. We first give some definitions in order to simplify the presentation.
\begin{defn}\label{defn2.3.1}
A pair of real numbers $(C, D)$ is called \emph{Type I} if both $C$ and $D$ are non-zero real numbers;
a pair of real matrices $(C, D)$ is called \emph{Type II} if the two matrices $(C,D)$ are of the form
\begin{equation}
\bigg(  \begin{pmatrix}
\alpha_1 & \beta_1\\
\beta_1 & -\alpha_1
\end{pmatrix},
  \begin{pmatrix}
\alpha_2 & \beta_2\\
\beta_2 & -\alpha_2
\end{pmatrix}
\bigg), \quad \text{ with }\alpha_2\beta_1-\alpha_1\beta_2> 0.
\end{equation}
It is called \emph{Standard Type II} if $\alpha_2\beta_1-\alpha_1\beta_2\equiv 1$.
\end{defn}

We then state the diagonalization result.
\begin{thm}\label{thm2.3.1}
Let $A_1,A_2$ be two non-singular real symmetric matrices. Assume that $A_1^{-1}A_2$ is diagonalizable over $\mathbb C$. Then there exists a non-singular real matrix $P$ such that
\begin{equation}\label{eq2.3.1}
  \begin{split}
P^tA_1P=\bar A_1=\text{diag}(C_1,\cdots,C_m),\\
P^tA_2P=\bar A_2=\text{diag}(D_1,\cdots,D_m),\\  
  \end{split}
\end{equation}
for some integer $m$ satisfying $1\leq m\leq n$, where the pair $(C_i,D_i)$ ($i=1,\cdots,m$) is either of 
\emph{Type I} or of \emph{Standard Type II}.
\end{thm}
\begin{rmk}\label{rmk2.3.2}
Note that we did not assume the positivity of the matrices $A_1$ and $A_2$, hence,
Theorem \ref{thm2.3.1} goes beyond the elementary result that if $A_1$ is symmetric positive-definite and $A_2$ is symmetric, then $A_1^{-1}A_2$ is diagonalizable over $\mathbb R$ and hence over $\mathbb C$. The argument for this classical result is as follows: if $A_1$ is symmetric positive-definite, then there exists $P$ such that $A_1=PP^t$, and hence $A_1^{-1}A_2 = P^{-t}P^{-1}A_2$, which is similar to $P^tP^{-t}P^{-1}A_2P^{-t} = P^{-1}A_2P^{-t}$. Then we find that $P^{-1}A_2P^{-t}$ is symmetric and of course diagonalizable over $\mathbb R$, and hence $A_1^{-1}A_2$ is diagonalizable over $\mathbb R$.
\end{rmk}
\begin{rmk}\label{rmk2.3.1}
The congruence diagonalization of two real symmetric matrices is already addressed in \cite{Uhl73}, see \cite[Corollary 1.4]{Uhl73} where the assumption is  that $A_1^{-1}A_2$ is diagonalizable over $\mathbb R$; see also \cite[Chapter 4]{HJ12} for additional related results. Our Theorem \ref{thm2.3.1} extends the congruence diagonalization by real matrix to the case where $A_1^{-1}A_2$ is diagonalizable over $\mathbb C$.
\end{rmk}
We need the following result in order to prove Theorem \ref{thm2.3.1} (see e.g. \cite[pp. 403]{LZ88}).
\begin{prop}\label{prop2.3.1}
Let $k_1$ and $k_2$ be two positive integers, and let $R$ be a $k_1\times k_2$ matrix, and $S_1$ and $S_2$ be two order-$k_2$ and order-$k_1$ square matrices respectively. Assume that $S_1$ and $S_2$ have no common eigenvalue, and $RS_1=S_2R$. Then $R=0$.
\end{prop}

\begin{proof}[Proof of Theorem \ref{thm2.3.1}]

Since $A_1^{-1}A_2$ is diagonalizable in $\mathbb C$, then by the standard diagonalization theory,
there exists a real non-singular matrix $P$ such that
\begin{equation}\label{eq2.3.11}
  P^{-1}A_1^{-1}A_2P=J=\text{diag}(J_1,\cdots,J_m),
\end{equation}
for some integer $m$ ($1\leq m\leq n$), where the $J_i$ ($i=1,\cdots,m$) are one of the following types:
\[
\begin{pmatrix}
  \lambda& &\\
  &\ddots& \\
  &&\lambda
\end{pmatrix}\quad
\text{ or }
\begin{pmatrix}
  E_0 &&\\
  &\ddots& \\
  &&E_0
\end{pmatrix}\text{ with } E_0: = \begin{pmatrix}
    \mu_1&-\mu_2\\
    \mu_2&\mu_1
  \end{pmatrix},
\]
where $\lambda$ is a real eigenvalue of $A_1^{-1}A_2$, and $\mu_1+i\mu_2$ ($\mu_2> 0$) is a complex eigenvalue of $A_1^{-1}A_2$. Furthermore, $J_i$ and $J_j$ have no common eigenvalue for $i\neq j$.

Multiplying both sides of \eqref{eq2.3.11} on the left by $P^tA_1P$ yields
\begin{equation}\label{eq2.3.12}
  P^tA_2P = P^tA_1PJ.
\end{equation}
The left-hand side of \eqref{eq2.3.12} is a symmetric matrix, we thus have that the right-hand side of \eqref{eq2.3.12} is also symmetric, which means that $P^tA_1PJ = J^tP^tA_1P$, that is
\begin{equation}\label{eq2.3.13}
 \begin{pmatrix}
  A_{11}&\cdots&A_{1m}\\
 \cdots &\cdots&\cdots\\
  A_{m1}&\cdots&A_{mm}\\
\end{pmatrix} \begin{pmatrix}
  J_1&&\\
  &\ddots&\\
  &&J_m
\end{pmatrix} = \begin{pmatrix}
  J_1^t&&\\
  &\ddots&\\
  &&J_m^t
\end{pmatrix}\begin{pmatrix}
  A_{11}&\cdots&A_{1m}\\
 \cdots &\cdots&\cdots\\
  A_{m1}&\cdots&A_{mm}\\
\end{pmatrix},
\end{equation}
where, similarly to the block diagonal matrix $J$, we write $P^tA_1P$ as:
\[
P^tA_1P = \begin{pmatrix}
  A_{11}&\cdots&A_{1m}\\
 \cdots &\cdots&\cdots\\
  A_{m1}&\cdots&A_{mm}\\
\end{pmatrix}.
\]
Comparing the two matrices on both sides of \eqref{eq2.3.13}, we find that $A_{ij}J_j=J_i^tA_{ij}$ for all $1\leq i,j\leq m$. Noticing that the $J_i$'s have no common eigenvalue and using Proposition~\ref{prop2.3.1}, we obtain that $A_{ij}=0$ for $i\neq j$, and hence that $P^tA_1P$ is block diagonal, that is
\begin{equation}\label{eq2.3.14}
  P^tA_1P = \text{diag}(A_{11},\cdots,A_{mm}).
\end{equation}
We then infer from \eqref{eq2.3.12} that $P^tA_2P$ is also block diagonal, that is
\begin{equation}\label{eq2.3.15}
  P^tA_2P = \text{diag}(A_{11}J_1,\cdots,A_{mm}J_m).
\end{equation}

By \eqref{eq2.3.14} and \eqref{eq2.3.15}, we see that it is enough to show that $A_{ii}$ and $A_{ii}J_i$ can be simultaneously diagonalized by congruence into the form \eqref{eq2.3.1} for all $i=1,\cdots,m$. 
We only need to consider the following three cases.

$i)$ Assume that $J_i$ is of form $\lambda I_k$, where the eigenvalue $\lambda\in\mathbb R$ is of multiplicity $k$. We know that there exists an orthogonal matrix $V_i$ such that $V_i^tA_{ii}V_i$ is diagonal. Moreover $V_i^t A_{ii}J_i V_i=\lambda V_i^t A_{ii} V_i$ which is also diagonal. Hence, $A_{ii}$ and $A_{ii}J_i$ can be simultaneously diagonalized by congruence into \emph{Type I}.

$ii)$ Assume that $J_i$ is of form $\displaystyle E_0=\begin{pmatrix}
    \mu_1&-\mu_2\\
    \mu_2&\mu_1
  \end{pmatrix}$, where the eigenvalue $\mu_1+i\mu_2$ ($\mu_2> 0$) is of multiplicity $1$. Since $A_{ii}$ is symmetric, we can assume that $\displaystyle
  A_{ii} = \begin{pmatrix}
\alpha_1 & \beta_1\\
\beta_1 & \tilde\alpha_1
\end{pmatrix}.
$
Since $A_{ii}J_i$ is also symmetric, we have
\begin{equation}\label{eq2.3.16}
 \begin{pmatrix}
\alpha_1 & \beta_1\\
\beta_1 & \tilde\alpha_1
\end{pmatrix}\begin{pmatrix}
    \mu_1&-\mu_2\\
    \mu_2&\mu_1
  \end{pmatrix}=\begin{pmatrix}
    \mu_1&\mu_2\\
    -\mu_2&\mu_1
  \end{pmatrix}\begin{pmatrix}
\alpha_1 & \beta_1\\
\beta_1 & \tilde\alpha_1
\end{pmatrix},
\end{equation}
which implies that $\alpha_1=-\tilde{\alpha_1}$. Direct computation shows that
\[
A_{ii}J_i=\begin{pmatrix}
\alpha_2 & \beta_2\\
\beta_2 & -\alpha_2
\end{pmatrix},
\]
where $\alpha_2= \alpha_1\mu_1+\beta_1\mu_2$, $\beta_2=\beta_1\mu_1-\alpha_1\mu_2$.
Notice that $\alpha_2\beta_1-\alpha_1\beta_2 = \mu_2(\alpha_1^2+\beta_1^2)>0$ since $\mu_2>0$. Hence, the pair $(A_{ii},\,A_{ii}J_i)$ is of \emph{Type II}. We now let $V_i=\text{diag}(\kappa_0,\kappa_0)$ with $\kappa_0=(\mu_2(\alpha_1^2+\beta_1^2))^{-1/4}$, so that
\begin{equation}\begin{split}
  V_i^tA_{ii}V_i = \kappa_0^2\begin{pmatrix}
\alpha_1 & \beta_1\\
\beta_1 & -\alpha_1
\end{pmatrix}=:\begin{pmatrix}
\bar\alpha_1 & \bar\beta_1\\
\bar\beta_1 & -\bar\alpha_1
\end{pmatrix},\\
  V_i^tA_{ii}J_iV_i = \kappa_0^2\begin{pmatrix}
\alpha_2 & \beta_2\\
\beta_2 & -\alpha_2
\end{pmatrix}=:\begin{pmatrix}
\bar\alpha_2 & \bar\beta_2\\
\bar\beta_2 & -\bar\alpha_2
\end{pmatrix},
\end{split}\end{equation}
and $\bar\alpha_2\bar\beta_1-\bar\alpha_1\bar\beta_2 = \kappa_0^4 \mu_2(\alpha_1^2+\beta_1^2) =1$.
Therefore, we find that $A_{ii}$ and $A_{ii}J_i$ can be simultaneously diagonalized by congruence into \emph{Standard Type II}.

$iii)$ Assume that $J_i$ is of form $\text{diag}(E_0,\cdots,E_0)_{2k\times 2k}$,
where the eigenvalue $\mu_1+i\mu_2$ ($\mu_2> 0$) is of multiplicity of $k$ ($k\geq 2$). Since $A_{ii}$ is symmetric, we can assume that
\[
A_{ii} = \begin{pmatrix}
  C_{11}&\cdots&C_{1k}\\
  &\cdots&\\
  C_{k1}&\cdots&C_{kk}
\end{pmatrix},
\]
where $C_{jj'}=C_{j'j}^t$ and $C_{jj'}$'s are $2\times 2$ matrices for all $1\leq j,j'\leq k$. Since $A_{ii}J_i$ is also symmetric, we thus find that
\begin{equation}\label{eq2.3.18}
  \begin{pmatrix}
  C_{11}&\cdots&C_{1k}\\
  &\cdots&\\
  C_{k1}&\cdots&C_{kk}
\end{pmatrix}\begin{pmatrix}
  E_0 &&\\
  &\ddots& \\
  &&E_0
\end{pmatrix}=\begin{pmatrix}
  E_0^t &&\\
  &\ddots& \\
  &&E_0^t
\end{pmatrix} \begin{pmatrix}
  C_{11}&\cdots&C_{1k}\\
  &\cdots&\\
  C_{k1}&\cdots&C_{kk}
\end{pmatrix}, 
\end{equation}
which, by comparing both sides of \eqref{eq2.3.18}, implies that $C_{jj'}E_0=E_0^tC_{jj'}$. That is
\begin{equation}\label{eq2.3.19}
 \begin{pmatrix}
\alpha & \tilde\beta\\
\beta & \tilde\alpha
\end{pmatrix}\begin{pmatrix}
    \mu_1&-\mu_2\\
    \mu_2&\mu_1
  \end{pmatrix}=\begin{pmatrix}
    \mu_1&\mu_2\\
    -\mu_2&\mu_1
  \end{pmatrix}\begin{pmatrix}
\alpha & \tilde\beta\\
\beta & \tilde\alpha
\end{pmatrix},
\end{equation}
where $\displaystyle
C_{jj'}=\begin{pmatrix}
\alpha & \tilde\beta\\
\beta & \tilde\alpha
\end{pmatrix}$. Since $\mu_2>0$, we then infer from \eqref{eq2.3.19} that $\tilde\beta=\beta$ and $\tilde\alpha=-\alpha$. Hence, the $C_{jj'}$ are all symmetric for $1\leq j,j'\leq k$, and we can conclude that
\begin{equation}\begin{split}\label{eq2.3.20}
A_{ii} = \begin{pmatrix}
  C_{11}&\cdots&C_{1k}\\
  &\cdots&\\
  C_{1k}&\cdots&C_{kk}
\end{pmatrix},\quad
A_{ii}J_i = \begin{pmatrix}
  D_{11}&\cdots&D_{1k}\\
  &\cdots&\\
  D_{1k}&\cdots&D_{kk}
\end{pmatrix}, 
\end{split}\end{equation}
where $D_{jj'} = C_{jj'}E_0$, and $C_{jj'}$ and $D_{jj'}$ are of the form  $\displaystyle
\begin{pmatrix}
\alpha &\beta\\
\beta & -\alpha
\end{pmatrix}$ for all $1\leq j,j'\leq k$. 
We are now in a position to prove that the pair ($A_{ii},A_{ii}J_i$) can be simultaneously diagonalized by congruence into \emph{Type II}, and we divide the proof into two steps.

\emph{Step I.} We first show that it is legitimate to assume that $C_{11}$ is non-singular. Indeed if $C_{11}$ is singular, we see that $C_{11}=0$ by its form $\displaystyle
\begin{pmatrix}
\alpha &\beta\\
\beta & -\alpha
\end{pmatrix}$. 
We now have two cases to consider.
\noindent\textbf{Case I:} If one of the $C_{jj}$ ($j=2,\cdots,k$) is non-singular, then without loss of generality, we assume that $C_{22}$ is non-singular. Then
\[
\begin{pmatrix}
0&I_2&\\
I_2&0&\\
&&I_{2k-4}
\end{pmatrix}
 \begin{pmatrix}
  C_{11}&C_{12}&\cdots&C_{1k}\\
  C_{12}&C_{22}&\cdots&C_{2k}\\
  &\cdots&\\
  C_{1k}&C_{2k}&\cdots&C_{kk}
\end{pmatrix}
\begin{pmatrix}
0&I_2&\\
I_2&0&\\
&&I_{2k-4}
\end{pmatrix}
= \begin{pmatrix}
  C_{22}&C_{12}&C_{23}&\cdots&C_{2k}\\
  C_{12}&C_{11}&C_{13}&\cdots&C_{1k}\\
  C_{23}&C_{13}&C_{33}&\cdots&C_{3k}\\
  &\cdots&\cdots&\cdots&\\
  C_{2k}&C_{1k}&C_{3k}&\cdots&C_{kk}
\end{pmatrix},
\]
where $I_2$ (resp. $I_{2k-4}$) is the identity matrix of order $2$ (resp. $2k-4$), and we see that $C_{22}$ is in the former position of $C_{11}$.
\textbf{Case II:} If all $C_{jj}$'s ($j=1,\cdots,k$) are singular, then one of $C_{1j}$ ($j=2,\cdots,k$) must be non-singular since $A_{ii}$ is non-singular. Without loss of generality, we assume that $C_{12}$ is non-singular, whereby
\begin{equation}\begin{split}\nonumber
\begin{pmatrix}
I_2&I_2&\\
-I_2&I_2&\\
&&I_{2k-4}
\end{pmatrix}
 \begin{pmatrix}
  0&C_{12}&\cdots&C_{1k}\\
  C_{12}&0&\cdots&C_{2k}\\
  &\cdots&\\
  C_{1k}&C_{2k}&\cdots&0
\end{pmatrix}
\begin{pmatrix}
I_2&-I_2&\\
I_2&I_2&\\
&&I_{2k-4}
\end{pmatrix}\\
=\begin{pmatrix}
  2C_{12}&0&C_{23}+C_{13}&\cdots&C_{2k} + C_{1k}\\
  0&-2C_{12}&C_{23}-C_{13}&\cdots&C_{2k}-C_{1k}\\
  C_{23}+C_{13}&C_{23}-C_{13}&C_{33}&\cdots&C_{3k}\\
  &\cdots&\cdots&\cdots&\\
  C_{2k} + C_{1k}&C_{2k} - C_{1k}&C_{3k}&\cdots&C_{kk}
\end{pmatrix},
\end{split}\end{equation}
and we see that $2C_{12}$ is in the former position of $C_{11}$.
Therefore, we can conclude that it is legitimate to assume that $C_{11}$ is non-singular under the congruence transformation. Once we have $C_{11}$ is non-singular, the corresponding $D_{11}$ is automatically non-singular since $D_{11}=C_{11}E_0$ and $E_0$ is non-singular.

\emph{Step II.} We now assume that $C_{11}$ is non-singular. Noticing that $D_{jj'}=C_{jj'}E_0$, we let
\[
V_i =\begin{pmatrix}
I_2&-C_{11}^{-1}C_{12}&\cdots&\cdots&-C_{11}^{-1}C_{1k}\\
&I_2&0&\cdots&0\\
&&\ddots&\ddots&\vdots\\
&&&I_2&0\\
&&&&I_2
\end{pmatrix}=\begin{pmatrix}
I_2&-D_{11}^{-1}D_{12}&\cdots&\cdots&-D_{11}^{-1}D_{1k}\\
&I_2&0&\cdots&0\\
&&\ddots&\ddots&\vdots\\
&&&I_2&0\\
&&&&I_2
\end{pmatrix}.
\]
Then direct computations (see Schur formula in \cite[pp. 150-151]{LZ88}) yields
\begin{equation}\begin{split}\label{eq2.3.23}
V_i^tA_{ii}V_i = \begin{pmatrix}
  C_{11}&0&\cdots&0\\
  0&\widetilde C_{22}&\cdots&\widetilde C_{2k}\\
  &\cdots&\cdots&\cdots&\\
  0&\widetilde C_{2k}&\cdots& \widetilde C_{kk}
\end{pmatrix},\quad
V_i^tA_{ii}J_iV_i  =  \begin{pmatrix}
  D_{11}&0&\cdots&0\\
  0&\widetilde D_{22}&\cdots&\widetilde D_{2k}\\
  &\cdots&\cdots&\cdots&\\
  0&\widetilde D_{2k}&\cdots& \widetilde D_{kk}
\end{pmatrix}, 
\end{split}\end{equation}
where $\widetilde C_{jj'} = C_{jj'} - C_{1j}C_{11}^{-1}C_{1j'}$ and $\widetilde D_{jj'} = D_{jj'} - D_{1j}D_{11}^{-1}D_{1j'}$ for all $2\leq j,j'\leq k$. We observe that all $\widetilde C_{jj'},\widetilde D_{jj'}$ are also of form  $\displaystyle
\begin{pmatrix}
\alpha &\beta\\
\beta & -\alpha
\end{pmatrix}$ for all $2\leq j,j'\leq k$. Therefore, by induction, we can obtain that the pairs $(A_{ii}, A_{ii}J_i)$ can be 
simultaneously diagonalized by congruence into \emph{Type II}. In order to diagonalize  $(A_{ii}, A_{ii}J_i)$ to the \emph{Standard Type II}, we use the same arguments as in case $ii)$. We thus completed the proof of Theorem \ref{thm2.3.1}.
\end{proof}
\section{An elliptic result}\label{sec-elliptic}

In this appendix, we prove an existence result for the first-order elliptic system in the domain $\Omega=(0,L_1)\times(0,L_2)$. Recall that $\Gamma_W,\Gamma_E,\Gamma_S,\Gamma_N$ are the boundaries $x=0,x=L_1,y=0,y=L_2$ respectively. 
We assume the following boundary conditions for all $j\in\aiminset{W,E,S,N}$:
\begin{equation}\label{eqa.1}
  a_j u_1 + b_j u_2 = 0,\text{ on }\Gamma_j,
\end{equation}
where $a_j,b_j$ are real constants such that $a_j^2+b_j^2 \neq 0$. Then the Lemmas 4.3.1.1-4.3.1.3 in \cite{Gri85} imply:
\begin{lemma}\label{lema.1}
The identity
$$\int_\Omega u_{2x}u_{1y} \text{d}x\text{d}y = \int_\Omega u_{1x}u_{2y} \text{d}x\text{d}y$$
holds for all $u=(u_1,u_2)^t\in H^1(\Omega)^2$ satisfying \eqref{eqa.1}.
\end{lemma}

Furthermore, we suppose that 
\begin{equation}\label{eqa.2}
 \text{ the $4\times 2$ matrix } ((a_j,b_j))_{ j\in\aiminset{W,E,S,N} }\text{ has full rank $2$.}
\end{equation}
The assumption \eqref{eqa.2} excludes the case where the functions on the four sides of $\partial\Omega$ are the same and thus guarantees the uniqueness of the solution in Theorem \ref{thma.3} below.
The main existence result is the following, which is an extension of \cite[Proposition 4.1]{HT12} to the non-constant coefficients case.
\begin{thm}\label{thma.1}
We assume that $\alpha_1,\alpha_2,\beta_1,\beta_2$ are $\mathcal C^{1,\gamma}(\overline\Omega)$-functions for some $0<\gamma<1$, and that
\[
\alpha_2\beta_1-\alpha_1\beta_2\geq c_0,
\]
for some constant $c_0>0$. Then for every given $\Psi=(\psi_1,\psi_2)^t\in \mathcal{C}_c^1(\Omega)^2$, there exists a unique solution $u\in H^1(\Omega)^2$ to the problem 
\begin{equation}\begin{cases}\label{eqa.3}
T_1 u_x+T_2 u_y=\Psi,\\
u_1 = 0,\text{ on }\Gamma_W\cup\Gamma_S=\aiminset{x=0}\cup\aiminset{y=0},\\
u_2 = 0,\text{ on }\Gamma_E\cup\Gamma_N=\aiminset{x=L_1}\cup\aiminset{y=L_2},
\end{cases}\end{equation}
where
\begin{equation*}
T_1=\begin{pmatrix}
\alpha_1 & \beta_1\\
\beta_1 & -\alpha_1
\end{pmatrix},\hspace{6pt}
T_2=\begin{pmatrix}
\alpha_2 & \beta_2\\
\beta_2 & -\alpha_2
\end{pmatrix}.
\end{equation*}
\end{thm}
\begin{proof}
We multiply by $T_1^{-1}$ on both sides of the equations in \eqref{eqa.3} and notice that
$\displaystyle
T_1^{-1}T_2 = \begin{pmatrix}
\mu_1&-\mu_2\\
\mu_2&\mu_1
\end{pmatrix}
$ with
$\displaystyle
\mu_1=\frac{\alpha_1\alpha_2+\beta_1\beta_2}{\alpha_1^2 + \beta_1^2},\,
\mu_2=\frac{\alpha_2\beta_1-\alpha_1\beta_2}{\alpha_1^2 + \beta_1^2}.
$
Since by assumption $\mu_2$ is positive away from zero, applying Theorem~\ref{thma.2} below yields the unique solution $u$ of \eqref{eqa.3}.
\end{proof}
\begin{rmk}\label{rmka.1}
Theorem \ref{thma.1} is still true if we choose any other boundary $\Gamma$ except $\emptyset$ and $\partial\Omega$, and is also true if we assume \eqref{eqa.2} and choose the boundary conditions \eqref{eqa.1}. 
The proof is exactly the same if we utilize Remark \ref{rmka.2} and Theorem \ref{thma.3} below.
\end{rmk}

\begin{thm}\label{thma.2}
Assume that $\mu_1,\mu_2$ are $\mathcal C^{1,\gamma}(\overline\Omega)$ functions for some $0<\gamma<1$, 
and that $\mu_2\geq c_0$
for some constant $c_0>0$. Then for every given $\Psi=(\psi_1,\psi_2)^t\in \mathcal{C}_c^1(\Omega)^2$, there exists a unique solution $u\in H^1(\Omega)^2$ to the problem 
\begin{equation}\begin{cases}\label{eqa.5}
  \begin{pmatrix}
    u_1\\
    u_2
  \end{pmatrix}_x + \begin{pmatrix}
\mu_1&-\mu_2\\
\mu_2&\mu_1
\end{pmatrix}\begin{pmatrix}
    u_1\\
    u_2
  \end{pmatrix}_y = \Psi,\\
  u_1 = 0,\text{ on }\Gamma=\Gamma_W\cup\Gamma_S,\\
  u_2 = 0,\text{ on }\Gamma^c=\Gamma_E\cup\Gamma_N,
\end{cases}\end{equation}
where $\Gamma^c$ is the complement of $\Gamma$ with respect to the boundary $\partial\Omega$.
\end{thm}
\begin{rmk}\label{rmka.2}
Theorem \ref{thma.2} is still true if we choose any other boundary $\Gamma$ except $\emptyset$ and $\partial\Omega$. The proof is exactly the same.
\end{rmk}
The generalization of Theorem \ref{thma.2} is the following:
\begin{thm}\label{thma.3}
Assume that $\mu_1,\mu_2$ satisfy the assumptions in Theorem \ref{thma.2} and \eqref{eqa.2} holds.
Then for every given $\Psi=(\psi_1,\psi_2)^t\in \mathcal{C}_c^1(\Omega)^2$, there exists a unique solution $u\in H^1(\Omega)^2$ to the problem 
\begin{equation}\begin{cases}\label{eqa.6}
  \begin{pmatrix}
    u_1\\
    u_2
  \end{pmatrix}_x + \begin{pmatrix}
\mu_1&-\mu_2\\
\mu_2&\mu_1
\end{pmatrix}\begin{pmatrix}
    u_1\\
    u_2
  \end{pmatrix}_y = \Psi,\\
  u\text{ satisfies the boundary condition \eqref{eqa.1}.}
\end{cases}\end{equation}
\end{thm}
We are now going to prove Theorems \ref{thma.2} and \ref{thma.3}. The idea of the proof is to use a new coordinate system (see \eqref{eqa.9}) to transform the equations in \eqref{eqa.5} and \eqref{eqa.6} to the non-homogeneous Cauchy-Riemann equations (see \eqref{eqa.13}), which furthermore can be reduced to a Laplacian equation (see \eqref{eqa.14}) with mixed (Dirichlet-Neumann) boundary conditions. Therefore, we can apply the results on elliptic problems in convex polygonal domain from \cite{Gri85} to obtain the solution $u\in H^1(\Omega)^2$ for \eqref{eqa.5} and \eqref{eqa.6}.
\begin{proof}[Proof of Theorem \ref{thma.2}]

We note that the equations in \eqref{eqa.5} are the non-homogeneous Beltrami equations. Hence, in order to show the existence of $u$ for \eqref{eqa.5}, we first study the homogeneous Beltrami equations, i.e.
\begin{equation}\begin{cases}\label{eqa.7}
\varphi_{1x} + \mu_1\varphi_{1y} - \mu_2 \varphi_{2y} =0,\\
\varphi_{2x} + \mu_1\varphi_{2y} + \mu_2\varphi_{1y}=0.
\end{cases}\end{equation}
In terms of the complex differential operators
\[
\partial_{\bar z}=(\partial_x + i\partial_y)/2,\quad \partial_{ z}=(\partial_x - i\partial_y)/2,
\]
and using the notation $w:=\varphi_1 + i\varphi_2$,  we can write \eqref{eqa.7} as 
\begin{equation}\label{eqa.8}
  w_{\bar z} = \frac{\mu_2 -1- i\mu_1}{\mu_2+1 - i\mu_1}w_z.
\end{equation}
This is the complex form of the Beltrami equations. We set $q_0 =\frac{\mu_2 -1- i\mu_1}{\mu_2+1 - i\mu_1}$, and since $\mu_2$ is positive away from zero in $\Omega$, we find that $\aiminabs{q_0}\leq \tau_0<1$ in $\Omega$ for some constant $\tau_0$. Since $\mu_2,\mu_1$ and hence $q_0$ are in $\mathcal C^{1,\gamma}(\overline\Omega)$, we can extend $q_0$ to $\mathbb C$ so that it is in $\mathcal C^{1,\gamma}(\mathbb C)$ and vanishes outside of some sufficiently large ball.
With this and the fact that $\aiminabs{q_0}\leq \tilde\tau_0<1$ for $\tau_0\leq\tilde\tau_0<1$, it can be shown that the Beltrami system \eqref{eqa.8} admits a solution $w\in \mathcal C^{2,\gamma}(\mathbb C)$ (see e.g. \cite{Sha45} or \cite[Chapter 2]{Vek62}, \cite[Chapter 4]{Hub06}). Furthermore, $w$ is a quasi-conformal mapping, i.e. $w$ preserves the orientation of the boundary of any bounded domain enclosed by a finite number of piecewise $\mathcal C^1$ curves. With the solutions $(\varphi_1,\varphi_2)$ for \eqref{eqa.7} at hand, we now introduce the new coordinate system $(x',y')$ such that
\begin{equation}\label{eqa.9}
  \begin{pmatrix}
    x'\\
    y'
  \end{pmatrix}=\begin{pmatrix}
    \varphi_1(x,y)\\
    \varphi_2(x,y)
  \end{pmatrix}.
\end{equation}
The transformation \eqref{eqa.9} is a valid coordinate transformation since by \eqref{eqa.7} the Jacobian matrix
\[
\frac{\partial(x',y')}{\partial(x,y)} = \begin{pmatrix}
  \varphi_{1x}&\varphi_{1y}\\
  \varphi_{2x}&\varphi_{2y}
\end{pmatrix} =\begin{pmatrix}
  \mu_2\varphi_{2y}-\mu_1\varphi_{1y} &\varphi_{1y}\\
  -\mu_2\varphi_{1y}-\mu_1\varphi_{2y}&\varphi_{2y}
\end{pmatrix}
\]
is non-singular, its determinant being equal to $\mu_2(\varphi_{1y}^2+\varphi_{2y}^2)$ and $\mu_2>0$. 

We denote by $\Gamma_j'$ the image of $\Gamma_j$ by this transformation for all $j\in\{W,E,S,N\}$, and denote by $\Omega', \Gamma',\theta',\Psi'$ and the gradient $\nabla'$ the transforms of $\Omega,\Gamma,\theta,\Psi$ and the gradient $\nabla$ respectively.
Now, direct computation gives
\begin{equation}\label{eqa.10}
\nabla u =\begin{pmatrix}
  \varphi_{1x}&\varphi_{2x}\\
  \varphi_{1y}&\varphi_{2y}
\end{pmatrix}\nabla' u'.
\end{equation}
In the new coordinate system $(x',y')$, the boundary conditions in \eqref{eqa.5} read
\begin{equation}\label{eqa.5p}
\begin{cases}
u_1' = 0 \text{ on } \Gamma', \\
u_2' = 0 \text{ on } \Gamma'^c,
\end{cases}
\end{equation}
where $\Gamma'=\Gamma_W'\cup\Gamma_S'$.
In the new coordinate system $(x',y')$, we rewrite the equation in \eqref{eqa.5} as
\begin{equation}\label{eqa.12}
  \begin{pmatrix}
   \mu_2\varphi_{2y} (u_{1x'}' - u_{2y'}') - \mu_2\varphi_{1y}(u_{2x'}' + u_{1y'}') \\
    \mu_2\varphi_{1y} (u_{1x'}' - u_{2y'}') + \mu_2\varphi_{2y}(u_{2x'}' + u_{1y'}') 
  \end{pmatrix} = \Psi',
\end{equation}
which yields
\begin{equation}\begin{cases}\label{eqa.13}
  u_{1x'}' - u_{2y'}' = f_1,\\
  u_{2x'}' + u_{1y'}' = f_2,
\end{cases}\end{equation}
for some functions $f_1,f_2$ which belong to $\mathcal C_c^1(\Omega')$ since $\Psi$ is given in $\mathcal C_c^1(\Omega)$.

Differentiating $\eqref{eqa.13}_1$ with respect to $x'$ and $\eqref{eqa.13}_2$ with respect to $y'$, and adding these two equations, we find the elliptic equation
\begin{equation}\label{eqa.14}
\Delta'u_1' = f_{1x'} + f_{2y'},
\end{equation}
where $\Delta'$ denotes the Laplace operator in the new coordinate system $(x',y')$. We associate with equation \eqref{eqa.14} the boundary conditions
\begin{equation}\label{eqa.15}
 u_1'=0 \text{ on } \Gamma',
 \end{equation}
which is already contained in \eqref{eqa.5p}$_1$. For the boundary $\Gamma'^c$, suitable boundary conditions can be obtained as follows. We first denote by $\nu_j'$ the unit normal vector to $\Gamma_j'$, and $\tau_j'$ the unit tangent vector on $\Gamma_j'$, for all $j\in\{W,E,S,N\}$. 
On $\Gamma_E=\aiminset{x=L_1}$, we first deduce from \eqref{eqa.5}$_1$ that $u_{1x}+\mu_1 u_{1y}=0$ since $u_2=0$ and $\Psi\in\mathcal C_c^1(\Omega)$, which implies that
\begin{equation}\label{eqa.16}
u_{1x'}'(\varphi_{1x}+\mu_1\varphi_{1y}) + u_{1y'}'(\varphi_{2x}+\mu_1\varphi_{2y}) = 0,\text{ on }\Gamma_E'.
\end{equation}
 Now, since $\Gamma_E'=\{(x',y')\,|\,(x',y')=(\varphi_1(L_1,y),\varphi_2(L_1,y)\}$, its tangent vector is parallel to $(\varphi_{1y},\varphi_{2y})$ or $(-\varphi_{2x}-\mu_1\varphi_{2y},\varphi_{1x} + \mu_1\varphi_{1y})$ (see \eqref{eqa.7}),
and its normal vector is parallel to $(\varphi_{1x} + \mu_1\varphi_{1y},\varphi_{2x} + \mu_1\varphi_{2y})$. Noticing \eqref{eqa.16}, 
 we then associate to \eqref{eqa.14} the following boundary condition on $\Gamma_E'$:
\begin{equation}\label{eqa.18}
\frac{\partial u_1'}{\partial\nu_E'}  = 0, \text{ on } \Gamma_E'.
\end{equation}
On $\Gamma_N'=\aiminset{ (x',y')\,|\,(x',y')=(\varphi_1(x,L_2),\varphi_2(x,L_2) }$, 
similar computations show that 
we need to associate to \eqref{eqa.14} the following boundary condition on $\Gamma_N'$:
\begin{equation}\label{eqa.19}
\frac{\partial u_1'}{\partial\nu_N'} = 0, \text{ on } \Gamma_N'.
\end{equation}

The existence and uniqueness of a solution $u_1'\in H^1(\Omega')$ of \eqref{eqa.14}-\eqref{eqa.15} and \eqref{eqa.18}-\eqref{eqa.19} follows from Lemma 4.4.3.1 of \cite{Gri85} with $\Omega=\Omega'$, $\mathscr{D} = \Gamma'$, $\mathscr{N} = \Gamma'^c$, $\beta_j=0$, and $f =  f_{1x'} + f_{2y'}$. 
We remark that a close look at the proof of Lemma 4.4.3.1 in \cite{Gri85} shows that it is still valid when the boundary of the domain is made of piecewise $\mathcal C^1$ curves instead of segments.

Following the arguments for $u_1'$, we find the following equation and boundary conditions for $u_2'$:
\begin{equation}\label{eqa.20}
\begin{cases}
\Delta'u_2' = -f_{1y'} + f_{2x'}, \\
u_2' = 0, \text{ on } \Gamma'^c=\Gamma_E'\cup\Gamma_N', \\
\frac{\partial u_2'}{\partial\nu_W'}  = 0,\text{ on } \Gamma_W', \\
\frac{\partial u_2'}{\partial\nu_S'} = 0,\text{ on } \Gamma_S'.
\end{cases}
\end{equation}
We thus also have a unique solution $u_2'\in H^1(\Omega')$ of \eqref{eqa.20} thanks to Lemma 4.4.3.1 of \cite{Gri85} again. 

In conclusion, in the new coordinate system $(x',y')$, we find a unique solution $ u'\in H^1(\Omega)^2$ which solves the problem \eqref{eqa.13} and \eqref{eqa.5p}. Transforming back to the original coordinate system $(x,y)$, we obtain $ u\in V$ satisfying \eqref{eqa.5}. Hence, the proof of Theorem~\ref{thma.2} is now complete.
\end{proof}

\begin{proof}[Proof of Theorem \ref{thma.3}]
We first prove the uniqueness. Noticing that $\mu_2$ is positive away from zero, we set
\begin{equation}\begin{split}
  Tu&=\begin{pmatrix}
    0&\frac{1}{\sqrt{\mu_2}}\\
    \frac{1}{\sqrt{\mu_2}}&0
  \end{pmatrix}\bigg(  \begin{pmatrix}
    u_1\\
    u_2
  \end{pmatrix}_x + \begin{pmatrix}
\mu_1&-\mu_2\\
\mu_2&\mu_1
\end{pmatrix}\begin{pmatrix}
    u_1\\
    u_2
  \end{pmatrix}_y\bigg)\\
  &=\begin{pmatrix}
    0&\frac{1}{\sqrt{\mu_2}}\\
    \frac{1}{\sqrt{\mu_2}}&0
  \end{pmatrix}\begin{pmatrix}
    u_1\\
    u_2
  \end{pmatrix}_x +\begin{pmatrix}
\sqrt{\mu_2}&\frac{\mu_1}{\sqrt{\mu_2}}\\
\frac{\mu_1}{\sqrt{\mu_2}}&-\sqrt{\mu_2}
\end{pmatrix}\begin{pmatrix}
    u_1\\
    u_2
  \end{pmatrix}_y,
\end{split}\end{equation}
and in order to prove the uniqueness, we only need to show that if $u\in H^1(\Omega)^2$ satisfies \eqref{eqa.1} and $Tu=0$, then $u=0$. Direct computations show that
\begin{equation}\begin{split}\label{eqa.22}
  \aiminnorm{Tu}_{L^2}^2 = &\int_\Omega \mu_2(u_{1y}^2+u_{2y}^2) + \frac{1}{\mu_2}\big( (u_{1x}+\mu_1u_{1y})^2 + (u_{2x}+\mu_1u_{2y})^2 \big)\,\text{d}x\text{d}y \\
  &+ \int_\Omega 2(u_{2x}u_{1y}-u_{1x}u_{2y})\,\text{d}x\text{d}y.
\end{split}\end{equation}
We first use Lemma \ref{lema.1} to dispense with the last term in the right-hand side of \eqref{eqa.22} since $u$ satisfies \eqref{eqa.1}, and then infer from $Tu=0$ and \eqref{eqa.22} that $\nabla u=0$ by noticing that $\mu_2$ is positive away from zero. Hence, $u$ is a constant function. The boundary conditions \eqref{eqa.1} and the assumption \eqref{eqa.2} impose that $u=0$ and the uniqueness follows.

We now prove the existence, and we use the same coordinate transformation \eqref{eqa.9} and the same notations introduced in the proof of Theorem \ref{thma.2}. The same arguments as in the proof of Theorem \ref{thma.2} will lead to the elliptic equation \eqref{eqa.14}, i.e.
\begin{equation}\label{eqa.23}
  \Delta'u_1' = f_{1x'} + f_{2y'}.
\end{equation}
Arguments similar to those used to find the boundary conditions for $u_1'$ in the proof of Theorem \ref{thma.2} lead to the suitable boundary conditions for \eqref{eqa.23}, for all $j\in\aiminset{W,E,S,N}$:
\begin{equation}\begin{split}\label{eqa.24}
  u_1' = 0, \,\text{ on }\Gamma_j, \text{ if }b_j=0,\\
  \frac{\partial u_1'}{\partial\nu_j'} + \frac{a_j}{b_j}\frac{\partial u_1'}{\partial\tau_j'} = 0, \,\text{ on }\Gamma_j, \text{ if }b_j\neq 0.\\
\end{split}\end{equation}
The existence of a (possibly non-unique) solution $u_1'\in H^1(\Omega')$ of \eqref{eqa.23}-\eqref{eqa.24} follows from Lemma 4.4.4.2 of \cite{Gri85} with $\Omega=\Omega'$ and $f=f_{1x'} + f_{2y'}$. We note that 
\begin{equation}\label{eqa.25}
\int_{\Omega'}(f_{1x'} + f_{2y'})\text{d}x'\text{d}y'=0,
\end{equation}
since $f_1,f_2$ vanish on the boundary $\partial\Omega'$. We need \eqref{eqa.25} to apply \cite[Lemma 4.4.4.2]{Gri85} in the case when $b_j\neq 0$ for all $j\in\aiminset{W,E,S,N}$.
We also remark that a close look at the proof of Lemma 4.4.4.2 in \cite{Gri85} shows that it is still valid when the boundary of the domain is made of piecewise $\mathcal C^1$ curves instead of segments.
The existence of a solution $u_2'\in H^1(\Omega')$ is similar.

In conclusion, in the new coordinate system $(x',y')$, we find a solution $u'=(u_1',u_2')^t\in H^1(\Omega')^2$, and transforming back to the original coordinate system $(x,y)$, we obtain a solution $u\in H^1(\Omega)^2$ satisfying \eqref{eqa.6}. Therefore, the proof of Theorem \ref{thma.3} is now complete.
\end{proof}

\section{An integration by parts formula}\label{sec-integration}
In this section, we are going to prove some results analogue to those proven in \cite[Section 1.3]{Tem01} for the Navier-Stokes equations. Consider the rectangular domain $\Omega$ where
\[
\Omega=(0,L_1)\times(0,L_2)
\]
for some $L_1,L_2>0$. Let $m$ be an integer, and $T_1=T_1(x,y)$ and $T_2=T_2(x,y)$ be two symmetric non-singular $m\times m$ matrices and belong to $\mathcal C^1(\overline\Omega)$. We consider the space
\[
\mathcal X(\Omega)=\aiminset{ \boldsymbol\theta\in L^2(\Omega)^m\;: (T_1\boldsymbol\theta)_x + (T_2\boldsymbol\theta)_y\in L^2(\Omega)^m }.
\]
We have an equivalent characterization of the space $\mathcal X(\Omega)$:
\[
\mathcal X(\Omega)=\aiminset{ \boldsymbol\theta\in L^2(\Omega)^m\;: T_1\boldsymbol\theta_x + T_2\boldsymbol\theta_y\in L^2(\Omega)^m }.
\]
The space $\mathcal X(\Omega)$ is endowed with the natural Hilbert norm $(\aiminnorm{\boldsymbol\theta}_{L^2}^2 + \aiminnorm{T_1\boldsymbol\theta_x + T_2\boldsymbol\theta_y}_{L^2}^2)^{1/2}$.

We aim to show that we can define a trace operator on the space $\mathcal X(\Omega)$. As a preliminary, we recall some results from \cite{Gri85}. Theorem 1.5.1.3 in \cite{Gri85} states that there exists a linear continuous operator (the trace operator) $\gamma_0 \in \mathcal L(H^1(\Omega), H^{1/2}(\Gamma))$ with $\Gamma = \partial\Omega$, and the trace operator $\gamma_0$ has a right continuous inverse operator (called lifting operator) $\ell_\Omega \in \mathcal L(H^{1/2}(\Gamma), H^1(\Omega))$ such that $\gamma_0\circ \ell_\Omega = $ the identity operator in $H^{1/2}(\Gamma)$. Let $H^{-1/2}(\Gamma)$ be the dual space of $H^{1/2}(\Gamma)$, we have the following trace theorem:
\begin{thm}\label{thmg.1}
Let $\Omega$ and $T_1,T_2$ be as above. Then there exists a linear continuous operator $\gamma_\nu\in \mathcal (\mathcal X(\Omega),H^{-1/2}(\Gamma))$ such that
\begin{equation}\label{eqg.1}
  \gamma_\nu \boldsymbol\theta =\text{the restriction of }T_1\boldsymbol\theta\nu_x + T_2\boldsymbol\theta\nu_y\text{ on }\Gamma,\quad\forall\,\boldsymbol\theta\in\mathcal D(\overline\Omega),
\end{equation}
where $\nu=(\nu_x,\nu_y)^t$ denotes the unit normal to $\Gamma$,
and the following integration by parts formula is true for all $\boldsymbol\theta\in\mathcal X(\Omega)$ and $\boldsymbol g\in H^1(\Omega)^m$:
\begin{equation}\label{eqg.2}
  \aimininner{ (T_1\boldsymbol\theta)_x + (T_2\boldsymbol\theta)_y }{\boldsymbol g} + \aimininner{T_1\boldsymbol g_x + T_2\boldsymbol g_y}{\boldsymbol\theta} = \aimininner{\gamma_\nu \boldsymbol\theta}{\gamma_0\boldsymbol g}_{H^{-1/2}(\Gamma)\times H^{1/2}(\Gamma)},
\end{equation}
where $\aimininner{\cdot}{\cdot}$ stands for the standard inner product on $L^2(\Omega)^m$.
\end{thm}
\begin{rmk}\label{rmkg.1}
With similar arguments as for the proof of Theorem \ref{thmg.1}, there also exists a linear continuous operator $\tilde\gamma_v\in \mathcal (\mathcal X(\Omega),H^{-1/2}(\Gamma))$ such that
\begin{equation}
  \tilde\gamma_\nu \boldsymbol\theta =\text{the restriction of }\boldsymbol\theta\nu_x +\boldsymbol\theta\nu_y \text{ on }\Gamma,\quad\forall\,\boldsymbol\theta\in\mathcal D(\overline\Omega),
\end{equation}
and the following integration by parts formula holds for all $\boldsymbol\theta\in\mathcal X(\Omega)$ and $\boldsymbol g\in H^1(\Omega)^m$:
\begin{equation}\label{eqg.3}
  \aimininner{ (T_1\boldsymbol\theta)_x + (T_2\boldsymbol\theta)_y }{\boldsymbol g} + \aimininner{T_1\boldsymbol g_x + T_2\boldsymbol g_y}{\boldsymbol\theta} = \aimininner{\tilde\gamma_\nu \boldsymbol\theta}{T_1\gamma_0\boldsymbol g + T_2\gamma_0\boldsymbol g}_{H^{-1/2}(\Gamma)\times H^{1/2}(\Gamma)}.
\end{equation}
\end{rmk}

In order to prove Theorem \ref{thmg.1}, we let $\boldsymbol\phi\in H^{1/2}(\Gamma)$ and let $\boldsymbol g\in H^1(\Omega)$ with $\gamma_0\boldsymbol g=\boldsymbol\phi$. For $\boldsymbol\theta\in\mathcal X(\Omega)$, we set
\begin{equation*}\begin{split}
X_{\boldsymbol\theta}(\boldsymbol\phi)&=\int_\Omega\big[((T_1\boldsymbol\theta)_x + (T_2\boldsymbol\theta)_y) \cdot \boldsymbol g  + (T_1\boldsymbol g_x + T_2\boldsymbol g_y)\cdot\boldsymbol\theta\big]\,dxdy\\
&=\aimininner{ (T_1\boldsymbol\theta)_x + (T_2\boldsymbol\theta)_y }{\boldsymbol g} + \aimininner{T_1\boldsymbol g_x + T_2\boldsymbol g_y}{\boldsymbol\theta}.
\end{split}\end{equation*}

\begin{lemma}\label{lemg.1}
$X_{\boldsymbol\theta}(\boldsymbol\phi)$ is independent of the choice of $\boldsymbol g$, as long as $\boldsymbol g\in H^1(\Omega)$ and $\gamma_0\boldsymbol g=\boldsymbol\phi$.
\end{lemma}
\begin{proof}
Let $\boldsymbol g_1$ and $\boldsymbol g_2$ belong to $H^1(\Omega)$ such that
\[
\gamma_0\boldsymbol g_1=\gamma_0\boldsymbol g_2=\boldsymbol\phi,
\]
and let $\boldsymbol g=\boldsymbol g_1- \boldsymbol g_2$. We must prove that
\[
\aimininner{ (T_1\boldsymbol\theta)_x + (T_2\boldsymbol\theta)_y }{\boldsymbol g_1} + \aimininner{T_1\boldsymbol g_{1,x} + T_2\boldsymbol g_{1,y}}{\boldsymbol\theta}=\aimininner{ (T_1\boldsymbol\theta)_x + (T_2\boldsymbol\theta)_y }{\boldsymbol g_2} + \aimininner{T_1\boldsymbol g_{2,x} + T_2\boldsymbol g_{2,y}}{\boldsymbol\theta},
\]
that is to say
\begin{equation}\label{eqg.4}
  \aimininner{ (T_1\boldsymbol\theta)_x + (T_2\boldsymbol\theta)_y }{\boldsymbol g} + \aimininner{T_1\boldsymbol g_x + T_2\boldsymbol g_y}{\boldsymbol\theta} = 0.
\end{equation}

But since $\boldsymbol g\in H^1(\Omega)$ and $\gamma_0\boldsymbol g=0$, $\boldsymbol g$ belongs to $H_0^1(\Omega)$ and is the limit in $H^1(\Omega)$ of smooth functions with compact support: $\boldsymbol g=\boldsymbol g_k$, $\boldsymbol g_k\in\mathcal D(\Omega)$.  It is obvious that
\[
  \aimininner{ (T_1\boldsymbol\theta)_x + (T_2\boldsymbol\theta)_y }{\boldsymbol g_k} + \aimininner{T_1\boldsymbol g_{k,x} + T_2\boldsymbol g_{k,y}}{\boldsymbol\theta} = 0,\quad\forall\,\boldsymbol g_k\in\mathcal D(\Omega),
\]
and \eqref{eqg.4} follows by taking the limit $k\rightarrow\infty$.
\end{proof}

\begin{proof}[Proof of Theorem \ref{thmg.1}]
For $\boldsymbol\phi\in H^{1/2}(\Gamma)$, let us take now $\boldsymbol g=\ell_\Omega \boldsymbol\phi$. Then by Schwarz inequality
\[
\aiminabs{ X_{\boldsymbol\theta}(\boldsymbol\phi) } \leq c_1 \aiminnorm{\boldsymbol\theta}_{\mathcal X(\Omega)}\aiminnorm{\boldsymbol g}_{H^1(\Omega)},
\]
where $c_1$ only depends on the norm of $T_1$ and $T_2$ in $\mathcal C^1(\overline\Omega)$. Since $\ell_\Omega\in\mathcal L(H^{1/2}(\Gamma), H^1(\Omega))$, we find
\begin{equation}\label{eqg.8}
  \aiminabs{ X_{\boldsymbol\theta}(\boldsymbol\phi) } \leq c_1c_0 \aiminnorm{\boldsymbol\theta}_{\mathcal X(\Omega)}\aiminnorm{\boldsymbol \phi}_{H^{1/2}(\Gamma)},
\end{equation}
where $c_0$ denotes the norm of the linear operator $\ell_\Omega$.

Therefore, the mapping $\boldsymbol\phi\mapsto X_{\boldsymbol\theta}(\boldsymbol\phi)$ is linear continuous mapping from $H^{1/2}(\Gamma)$ into $\mathbb R$. Thus there exists $\boldsymbol h=\boldsymbol h(\boldsymbol\theta)\in H^{-1/2}(\Gamma)$ such that
\begin{equation}\label{eqg.9}
  X_{\boldsymbol\theta}(\boldsymbol\phi) = \aimininner{\boldsymbol h(\boldsymbol\theta)}{\boldsymbol\phi}.
\end{equation}
It is clear that the mapping $\boldsymbol\theta\mapsto \boldsymbol h(\boldsymbol\theta)$ is linear, and by \eqref{eqg.8},
\begin{equation}
  \aiminnorm{\boldsymbol h}_{H^{-1/2}(\Gamma)} \leq c_1c_0\aiminnorm{\boldsymbol\theta}_{\mathcal X(\Omega)};
\end{equation}
this proves that the mapping $\boldsymbol \theta\mapsto \boldsymbol h(\boldsymbol \theta)=\gamma_\nu\boldsymbol\theta$ is continuous from $\mathcal X(\Omega)$ into $H^{-1/2}(\Gamma)$.

The last point to prove \eqref{eqg.1} and \eqref{eqg.2} follows from the definition of $\gamma_\nu\boldsymbol\theta$ and Lemma \ref{lemg.2} below.
\end{proof}
\begin{lemma}\label{lemg.2}
If $\boldsymbol\theta\in\mathcal D(\overline\Omega)$. Then
\[
\gamma_\nu\boldsymbol\theta = \text{the restriction of }T_1\boldsymbol\theta\nu_x + T_2\boldsymbol\theta\nu_y\text{ on }\Gamma,
\]
where $\nu=(\nu_x,\nu_y)^t$ denotes the unit normal to $\Gamma$.
\end{lemma}
\begin{proof}
For such a smooth $\boldsymbol\theta$ and for any $\boldsymbol g\in\mathcal D(\overline\Omega)$, we have
\begin{equation*}\begin{split}
 X_{\boldsymbol\theta}(\gamma_0\boldsymbol g) &=  \int_\Omega\big[((T_1\boldsymbol\theta)_x + (T_2\boldsymbol\theta)_y) \cdot \boldsymbol g  + (T_1\boldsymbol g_x + T_2\boldsymbol g_y)\cdot\boldsymbol\theta\big]\,dxdy\\
&=\int_\Omega \text{div}\,(\boldsymbol g\cdot T_1\boldsymbol\theta,\, \boldsymbol g\cdot T_2\boldsymbol\theta)\,dxdy\\
&=\int_\Gamma (\boldsymbol g\cdot T_1\boldsymbol\theta,\, \boldsymbol g\cdot T_2\boldsymbol\theta)\cdot \nu\,d\Gamma\quad\text{ (by the Stokes formula)}\\
&=\aimininner{ T_1\boldsymbol\theta\nu_x + T_2\boldsymbol\theta\nu_y }{\gamma_0\boldsymbol g},
\end{split}\end{equation*}
where $\nu=(\nu_x,\nu_y)^t$ is the unit normal to $\Gamma$

Since for these function $\boldsymbol g$, the traces $\gamma_0\boldsymbol g$ form a dense subset of $H^{1/2}(\Gamma)$, the formula
\[
 X_{\boldsymbol\theta}(\boldsymbol\phi) = \aimininner{ T_1\boldsymbol\theta\nu_x + T_2\boldsymbol\theta\nu_y }{\boldsymbol\phi}
\]
is also true by continuity for every $\boldsymbol \phi\in H^{1/2}(\Gamma)$. By comparison with \eqref{eqg.9}, we obtain that $\gamma_\nu\boldsymbol\theta=T_1\boldsymbol\theta\nu_x + T_2\boldsymbol\theta\nu_y$.
\end{proof}
\section{The density theorems}\label{sec-density}
In this appendix, we establish general density theorems regarding function spaces defined on the domain $\Omega=(0,L_1)\times(0,L_2)$. We first recall the results from \cite[Section 2.2]{HT12} in the constant coefficients case and then extend those results to the variable coefficients case. These theorems were needed for proving the positivity of certain unbounded operators defined in Subsections \ref{subsec2.1},\,\ref{subsec3.1} and showing that they are infinitesimal generators of (quasi-)contraction semigroups on Lebesgue spaces.

\subsection{The constant coefficients case}\label{sec-density-constant}
All the results in this subsection are taken from \cite[Section 2.2]{HT12}, and we briefly state them as follows.

For $\lambda$ fixed, $\lambda\in \mathbb{R}, \lambda\neq 0$, we set $T\theta=\theta_y+\lambda \theta_x$, and introduce the function space
\begin{equation*}
\mathcal{X}_1(\Omega)=\{ \theta\in L^2(\Omega), T\theta=\theta_y+\lambda \theta_x\in L^2(\Omega) \}. 
\end{equation*}
We first have the following density result.
\begin{prop}\label{prop1}
$\mathcal{C}^\infty(\overline\Omega)\cap\mathcal{X}_1(\Omega)$ is dense in $\mathcal{X}_1(\Omega)$.
\end{prop}
The proof of Proposition \ref{prop1} is classically conducted by using the method of partition of unity, and is simpler than the proof of Proposition \ref{T-prop3.1} below for the variable coefficients case, where we will give full details.

We also have the following trace result.
\begin{prop}\label{M-propb.2}
If $u\in \mathcal{X}(\Omega)$, then the traces of $u$ are defined on all of $\partial\Omega$, i.e. the traces of $u$ are defined at $x=0,L_1$, and $y=0,L_2$, and they belong to the respective spaces $H_y^{-1}(0,L_2)$ and $H_x^{-1}(0,L_1)$. Furthermore the trace operators are linear continuous in the corresponding spaces, e.g., $u\in\mathcal{X}(\Omega)\rightarrow u|_{x=0}$ is continuous from $\mathcal{X}(\Omega)$ into $H_y^{-1}(0,L_2)$.
\end{prop}

Recall that $\Gamma_W,\Gamma_E,\Gamma_S,\Gamma_N$ are the boundaries $x=0,x=L_1,y=0,y=L_2$ respectively, and $\Gamma$ can be any union of the sets $\Gamma_W,\Gamma_E,\Gamma_S,\Gamma_N$. 
For any function $v$ defined on $\Omega$, here and again in the following we denote by $\tilde v$ the function equal to $v$ in $\Omega$ and to $0$ in $\mathbb{R}^2\backslash\Omega$. 
We introduce the function spaces:
\begin{equation}\label{eq:eq01}
\mathcal{X}_\Gamma(\Omega)=\{ \theta\in L^2(\Omega), T\theta=\theta_y+\lambda \theta_x\in L^2(\Omega), \theta|_\Gamma = 0 \},
\end{equation}
\begin{equation}\nonumber
\mathcal{V}_\Gamma(\Omega) = \{ \theta\in \mathcal{C}^\infty(\overline{\Omega}), 
\text{ and $\theta$ vanishes in a neighborhood of } \Gamma\}.
\end{equation}

The main density result of \cite[Section 2.2]{HT12} is the following.
\begin{thm}\label{L-thm1}
Suppose that $\Gamma=\Gamma_W\cup\Gamma_S$ and $\lambda>0$. Then
\begin{equation}\nonumber
\mathcal{V}_\Gamma(\Omega)\cap\mathcal{X}_\Gamma(\Omega) \text{ is dense in } \mathcal{X}_\Gamma(\Omega).
\end{equation}
\end{thm}

\begin{rmk}\label{L-rmk2}
Theorem \ref{L-thm1} is also valid when $\Gamma$ is made of two contiguous sides of $\partial\Omega$, that is in the following three cases:
\begin{equation}
\begin{cases}
\Gamma=\Gamma_E\cup\Gamma_N \text{ and } \lambda >0, \\
\Gamma=\Gamma_W\cup\Gamma_N \text{ and } \lambda <0, \\
\Gamma=\Gamma_E\cup\Gamma_S \text{ and } \lambda <0.
\end{cases}
\end{equation}
\end{rmk}

\subsection{The variable coefficients case}\label{T-density}
In this subsection, we choose $\lambda=\lambda(x,y)\in\mathcal C^1(\overline\Omega)$ satisfying 
\begin{equation}\label{T-asp3.1}
\begin{cases}
c_0 \leq \lambda(x,y) \leq c_1, \\
\aiminabs{\lambda_x} \leq M,
\end{cases}
\end{equation}
where $c_0,c_1,M$ are positive constants. We extend $\lambda$ to $\mathbb R^2$ so that $\lambda$  belongs to $\mathcal C^1(\mathbb R^2)$.
We set $T\theta=\lambda(x,y) \theta_x + \theta_y$, and introduce the function space
\begin{equation}\nonumber
\mathcal{Y}_1(\Omega)=\{ \theta\in L^2(\Omega), T\theta=\lambda(x,y) \theta_x + \theta_y\in L^2(\Omega) \}.
\end{equation}
We observe that $\mathcal{Y}_1(\Omega)$ is a space of local type, that is
\begin{equation}\label{S-eq3.e1}
 \text{If }\theta\in \mathcal{Y}_1(\Omega), \psi\in\mathcal{C}^\infty(\overline\Omega), \text{ then } \theta\psi\in\mathcal{Y}_1(\Omega).
\end{equation}
This property follows from $T(\psi\theta)=\psi T\theta+(\lambda\psi_x+\psi_y)\theta$.

We need to prove results similar to those of Section \ref{sec-density-constant}. 
We first show that the smooth functions are dense in $\mathcal{Y}_1(\Omega)$. 
\begin{prop}\label{T-prop3.1}
$\mathcal{C}^\infty(\overline\Omega)\cap\mathcal{Y}_1(\Omega)$ is dense in $\mathcal{Y}_1(\Omega)$.
\end{prop}
\begin{proof}
Using a proper covering of $\Omega$ by sets $\mathcal{O}_0,\mathcal{O}_1,\cdots,\mathcal{O}_N$, we consider a partition of unity subordinated to this covering, $1=\sum_{i=0}^N \psi_i$. Here and again in this section we will use a covering of $\Omega$ consisting of $\mathcal{O}_0$, a relatively compact subset of $\Omega$, and of sets $\mathcal{O}_i$ of one of the following types: $\mathcal{O}_i$ is a ball centered at one of the corners of $\Omega$, which does not intersect the two other sides of $\Omega$; or $\mathcal{O}_i$ is a ball centered on one of the sides of $\Omega$ which does not intersect any of the three other sides of $\Omega$.

If $\theta\in\mathcal{Y}_1(\Omega)$, then $\theta\psi_i\in\mathcal{Y}_1(\Omega)$ by \eqref{S-eq3.e1}, so that we only need to approximate $\theta\psi_i$ by smooth functions. Here the support of $\psi_i$ is contained in the set $\mathcal{O}_i$, and we start with considering the set $\mathcal{O}_0$, relatively compact in $\Omega$, then we consider the balls $\mathcal{O}_i$ centered on the boundary $\partial\Omega$.

We first consider the case where $\psi_i=\psi_0$ and $\mathcal{O}_i =\mathcal{O}_0$ which is relatively compact in $\Omega$.
Let $\rho$ be a mollifier such that $\rho\geq 0, \int\rho=1$, and $\rho$ has compact support. 

$i)$ The function $v=\theta\psi_0\in\mathcal{Y}_1(\Omega)$ has compact support in $\mathcal{O}_0$. Since $\mathcal{O}_0$ is relatively compact in $\Omega$, then for $\epsilon$ small enough, $\rho_\epsilon*v$ is supported in $\Omega$. Then the standard mollifier theory shows that for $\epsilon\rightarrow 0$:
\begin{equation}\label{S-eq05}
\begin{cases}
\rho_\epsilon* \tilde v \rightarrow \tilde v,\hspace{6pt} \text{in }L^2(\mathbb R^2),\\
\rho_\epsilon*\widetilde{T v}\rightarrow \widetilde{T v},\hspace{6pt} \text{in }L^2(\mathbb R^2).
\end{cases}
\end{equation}
Since the convolution and the operator $T$ do not commute in the non-constant coefficient case, we need the following Friedrichs' lemma (see \cite{Fri44} or \cite[Theorem 3.1]{Hor61}).
\begin{lemma}\label{S-lem3.2extra}
Let ${\mathcal U}$ be an open set of $\mathbb{R}^d$.
If $\nabla a\in L^\infty(\mathcal U)$ and $u\in L_{\text{loc}}^2(\mathcal U)$, then for all $1\leq j\leq d$,
\begin{equation}\nonumber
a \partial_{x_j}(u*\rho_\epsilon) - (a\partial_{x_j} u)*\rho_\epsilon \rightarrow 0, \text{ when }\epsilon\rightarrow 0,
\end{equation}
in the sense of $L^2$ convergence on all compact subsets of $\mathcal{U}$.
\end{lemma}
We then continue the proof of Proposition \ref{T-prop3.1}.
Noting that $v=\theta\psi_0$ has compact support in $\Omega$, we apply Lemma \ref{S-lem3.2extra} with $\mathcal{U}=\Omega, a=\lambda$ and $u= v$; we see that, as $\epsilon\rightarrow 0$,  
\begin{equation}\label{S-eq06extra2}
T(\rho_\epsilon* v) - \rho_\epsilon*T v \rightarrow 0,
\end{equation}
in $L^2(K)$ for any compact set $K\subset\Omega$ and thus in $L^2(\Omega)$ since $v$ is compactly supported in $\Omega$.
Combining \eqref{S-eq05}$_2$ and \eqref{S-eq06extra2}, we obtain that as $\epsilon\rightarrow 0$,
\begin{equation}\label{S-eq06}
T(\rho_\epsilon* v) \rightarrow T v,\,\text{in }L^2(\Omega).
\end{equation}
Therefore, $v_\epsilon=(\rho_\epsilon*v)|_{\Omega}$ converges to $v$ in $\mathcal{Y}_1(\Omega)$ by $\eqref{S-eq05}_1$ and \eqref{S-eq06}.

 $ii)$ We then consider the case where $\psi_i=\psi_1$, and $\mathcal{O}_i =\mathcal{O}_1$ which is a ball centered at the origin $(0,0)$; the other cases are similar or simpler. Let $\rho$ be a mollifier as before, but now $\rho$ is compactly supported in $\{x<0, y<0\}$. Then for $v=\theta\psi_1$, we observe that
\begin{equation}\label{S-eq06-7}
T\tilde{v} = \widetilde{T v} + \mu,
\end{equation}
where $\mu$ is a measure supported by $\aiminset{x=0}\cup\aiminset{y=0}$.
Then mollifying \eqref{S-eq06-7} with this $\rho$ gives
\begin{equation}\label{S-eq07}
\rho_\epsilon*(T\tilde v) =\rho_\epsilon*\widetilde{T v} + \rho_\epsilon*\mu.
\end{equation}
By the choice of the support of $\rho$, $\rho_\epsilon*\mu$ is supported outside of $\Omega$. Hence, restricting \eqref{S-eq07} to $\Omega$ implies that:
\begin{equation}\label{S-eq07extra1}
(\rho_\epsilon*(T\tilde v))\big|_{\Omega} =\rho_\epsilon*(\widetilde{T v})\big|_{\Omega} \rightarrow T v, \text{ in }L^2(\Omega),\text{ as }\epsilon\rightarrow 0.
\end{equation}
Applying Lemma \ref{S-lem3.2extra} with $\mathcal{U}=\mathbb{R}^2, a=\lambda$ and $u=\tilde v$, we obtain that as $\epsilon\rightarrow 0$,
\begin{equation}
T(\rho_\epsilon*\tilde v) - \rho_\epsilon*T\tilde v \rightarrow 0,\text{ in }L^2(\Omega),
\end{equation}
which, combined with \eqref{S-eq07extra1}, implies that
\begin{equation}\label{S-eq07extra2}
T(\rho_\epsilon*\tilde v)\big|_{\Omega} \rightarrow T v, \text{ in }L^2(\Omega),\text{ as }\epsilon\rightarrow 0,
\end{equation}
If we set $\tilde v_\epsilon=\rho_\epsilon*\tilde v$, then as $\epsilon\rightarrow 0$, $\tilde v_\epsilon\rightarrow\tilde v$ in $L^2(\mathbb{R}^2)$, and
\begin{equation}\label{S-eq08}
\begin{cases}
\tilde v_\epsilon\big|_\Omega \rightarrow v, \text{ in } L^2(\Omega);\\
T(\tilde v_\epsilon\big|_\Omega)\rightarrow T v, \text{ in } L^2(\Omega),
\end{cases}
\end{equation}
which shows that $\tilde v_\epsilon\big|_\Omega$ converges to $v$ in $\mathcal{Y}_1(\Omega)$.
\end{proof}

We are now going to prove the density theorems involving the boundary $\partial\Omega$. Recall that $T\theta=\lambda(x,y) \theta_x+\theta_y$, and we introduce the function spaces:
\begin{equation}\nonumber
\mathcal{Y}_\Gamma(\Omega)=\{ \theta\in L^2(\Omega), T\theta=\lambda(x,y) \theta_x+\theta_y\in L^2(\Omega), \theta|_\Gamma = 0 \}.
\end{equation}
Using the same arguments as in \cite[Proposition 2.3]{HT12}, we see that for all $i\in\aiminset{W,E,S,N}$, the traces on $\Gamma_i$'s are well defined for the functions belonging to $\mathcal{Y}_\Gamma(\Omega)$ since $\lambda(x,y)$ is positive away from zero. We then state the density theorem:
\begin{thm}\label{T-thm3.1}
Suppose that $\Gamma=\Gamma_W\cup\Gamma_S$ and $\lambda(x,y)$ satisfies \eqref{T-asp3.1} Then 
\begin{equation}\nonumber
\mathcal{V}_\Gamma(\Omega)\cap\mathcal{Y}_\Gamma(\Omega) \text{ is dense in } \mathcal{Y}_\Gamma(\Omega).
\end{equation}
\end{thm}

Following the same arguments as in \cite[Theorem 2.1]{HT12} and utilizing Lemma \ref{S-lem3.2extra}, we can obtain Theorem \ref{T-thm3.1}.

\begin{rmk}\label{T-rmk3.2}
Looking carefully at the proof of Theorem \ref{T-thm3.1} (see actually the proof of \cite[Theorem 2.1]{HT12}), we see that  
Theorem \ref{T-thm3.1} is also valid if $\Gamma=\Gamma_E\cup\Gamma_N$. Moreover, if the assumptions on $\lambda(x,y)$ are
\begin{equation}\tag{\ref{T-asp3.1}$'$}\label{T-asp3.1prime}
\begin{cases}
-c_1 \leq \lambda(x,y) \leq -c_0, \\
\aiminabs{\lambda_x} \leq M,
\end{cases}
\end{equation}
where $c_0,c_1,M$ are positive constants, then Theorem \ref{T-thm3.1} is still true if $\Gamma$ is $\Gamma_W\cup\Gamma_N$ or $\Gamma_E\cup\Gamma_S$.
\end{rmk}

\section{Preliminary results about semigroups}\label{sec-semigroup}
In this appendix we collects some basic facts on the semigroups and the characterization of their generators, and also prove some useful results about the semigroups on Hilbert space. The main references are the classical books by K. Yosida \cite{Yos80}, and by E. Hille and R.S. Phillips \cite{HP74}, and by A. Pazy \cite{Paz83} and the book by K.-J. Engel and R. Nagel \cite{EN00}.
\begin{defn}
A family $(S(t))_{t\geq 0}$ of bounded linear operators on a Banach space $X$ is called a strongly continuous (one-parameter) semigroup (or $\mathcal{C}_0$-semigroup) if it satisfies 
\begin{enumerate}[i)]
\item $S(0)=I,S(t+s)=S(t)S(s)$ for all $t,s\geq 0$;
\item $\xi_x : t\mapsto \xi_x(t):=S(t)x$ is continuous from $\mathbb{R}_+$ into $X$ for every $x\in X$.
\end{enumerate}
\end{defn}

\begin{prop}\label{S-propa.1}
For every strongly continuous semigroup $(S(t))_{t\geq 0}$, there exist constants $\omega\in\mathbb{R}$ and $M\geq 1$ such that
\begin{equation}\label{S-eqa.1}
\aiminnorm{S(t)}\leq Me^{\omega t}
\end{equation}
for all $t\geq 0$.
\end{prop}

\begin{defn}\label{S-defn2.2}
A strongly continuous semigroup is called a \emph{quasi-contraction} if we can take $M=1$ in \eqref{S-eqa.1}, and  called \emph{bounded} if $\omega=0$, and called \emph{contraction} if $\omega=0$ and $M=1$ is possible. 
\end{defn}
	%The adjoint semigroup of a strongly continuous semigroup on a reflexive Banach space is again strongly continuous.

\begin{defn}
The generator $A:\mathcal{D}(A)\subset X\mapsto X$ of a strongly continuous semigroup $(S(t))_{t\geq 0}$ on a Banach space $X$ is the operator
\begin{equation}\nonumber
Ax:=\dot{\xi}_x(0)=\lim_{h\downarrow 0}\frac{1}{h}(S(h)x -x)
\end{equation}
defined for every $x$ in its domain
\begin{equation}\nonumber
\mathcal{D}(A):=\{x\in X\,:\,t\mapsto\xi_x(t) \text{ is right differentiable in }t\text{ at }t=0\}.
\end{equation}
\end{defn}
Note that if $\xi_x(t)$ is right differentiable in $t$ at $t=0$, it is differentiable at $t$ for any $t>0$.

In the following, we only consider the case when the Banach space $X$ is a Hilbert space $H$.

\begin{defn}
A linear operator $(A,\mathcal{D}(A))$ on a Hilbert space $H$ is called \emph{positive} if
\begin{equation}
\aimininner{Ax}{x}\geq 0,
\end{equation}
for all $x\in\mathcal{D}(A)$.
\end{defn}

\begin{thm}[Hille-Yosida theorem]\label{S-thm2.1}
Let $(A,\mathcal{D}(A))$ be a closed, densely defined operator on a Hilbert space $H$. If $A$ is positive and
the operator $\omega+A$ is surjective for some (hence for all) $\omega>0$. Then $-A$ generates a contraction semigroup.
\end{thm}

\begin{thm}[Hille-Phillips-Yosida theorem]\label{S-thm2.2}
Let $(A,\mathcal{D}(A))$ be a closed, densely defined operator on a Hilbert space $H$. If both $A$ and its adjoint $A^*$ are positive, then $-A$ generates a contraction semigroup on $H$.
\end{thm}
We remark that if $(A,\mathcal{D}(A))$ is a closed, densely defined operator on $H$, then its adjoint $A^*$ is also closed, densely defined.

	%\item The operator $A$ generates a contraction semigroup on $X$.
	%\item The range of $\omega - A$ is dense in $X$ for some (hence all) $\omega>0$.
\begin{prop}\label{S-propa.3}
Let $(A,\mathcal{D}(A))$ be a closed, densely defined operator on a Hilbert space $H$, and assume that 
$-A$ generates a contraction semigroup on $H$. Then both $A$ and its adjoint $A^*$ are positive.
\end{prop}
The proof of Proposition \ref{S-propa.3} can be found in \cite[pp. 88]{EN00}.

\begin{thm}[Bounded Perturbation Theorem]\label{S-thm2.3}
Let $(A,\mathcal{D}(A))$ be the infinitesimal generator of a strongly continuous semigroup $(S(t))_{t\geq 0}$ on a Banach space $X$ satisfying
$$\aiminnorm{S(t)}\leq M_0 e^{\omega t},\,\forall\, t\geq 0, $$
where $\omega\in \mathbb{R}, M_0\geq 1$. If $B\in \mathcal{L}(X)$, then
$C:=A+B$, with $\mathcal{D}(C):=\mathcal{D}(A)$
generates a strongly continuous semigroup $(R(t))_{t\geq 0}$ satisfying
\begin{equation}\nonumber
\aiminnorm{R(t)}\leq M_0 e^{(\omega + M_0\aiminnorm{B})t},\,\forall\, t\geq 0.
\end{equation}
\end{thm}

\begin{thm}\label{S-thm2.4}
Let $(A,\mathcal{D}(A))$ be a closed, densely defined operator on a Hilbert space $H$. If $A$ is quasi-positive in the sense that
\[
\aimininner{ \mathcal Au}{u}_H \geq -\omega_0\aiminnorm{u}_H^2,\hspace{6pt}\forall\, u\in\mathcal{D}(\mathcal A),
\]
for some $\omega_0>0$, and the operator $\omega+A$ is surjective for some (hence for all) $\omega>\omega_0$. Then $-A$ generates a quasi-contraction semigroup on $H$.
\end{thm}
\begin{proof}
We set $B=\omega_0+A$, with $\mathcal{D}(B)=\mathcal{D}(A)$, then $B$ is a positive operator such that $\omega+B$ is surjective for some (hence for all) $\omega>0$. Therefore, Theorem \ref{S-thm2.1} implies that the operator $-B$ generates a contraction semigroup on $H$, and we conclude the result by using the Bounded Perturbation Theorem \ref{S-thm2.3}.
\end{proof}

\begin{thm}\label{S-thm2.5}
Let $(A,\mathcal{D}(A))$ be a closed, densely defined operator on a Hilbert space $H$. If both $A$ and its adjoint $A^*$ are quasi-positive in the sense that there exists a constant $\omega_0>0$ such that
\begin{equation}\label{T-eqb.1}
\begin{cases}
\aimininner{ \mathcal Au}{u}_H \geq -\omega_0\aiminnorm{u}_H^2,\hspace{6pt}\forall\, u\in\mathcal{D}(\mathcal A),\\
\aimininner{ \mathcal A^*u}{u}_H \geq -\omega_0\aiminnorm{u}_H^2,\hspace{6pt}\forall\, u\in\mathcal{D}(\mathcal A^*).
\end{cases}
\end{equation}
Then $-A$ generates a quasi-contraction semigroup on $H$.
\end{thm}
\begin{proof}
We set $B=\omega_0+A$, with $\mathcal{D}(B)=\mathcal{D}(A)$, and the adjoint $B^*=\omega_0+A^*$, with $\mathcal{D}(B^*)=\mathcal{D}(A^*)$. Then both $B$ and $B^*$ are positive operators on $H$ by virtue of \eqref{T-eqb.1}. It is clear that both $B$ is also a closed, densely defined operator. Therefore, Theorem \ref{S-thm2.2} implies that the operator $-B$ generates a contraction semigroup on $H$, and we conclude the result by using the Bounded Perturbation Theorem \ref{S-thm2.3}.
\end{proof}

\begin{prop}\label{S-prop2.6}
Let $(A,\mathcal{D}(A))$ be a closed, densely defined operator on a Hilbert space $H$. If $-A$ generates a quasi-contraction semigroup $(S(t))_{t\geq 0}$ on $H$, i.e. $\aiminnorm{S(t)}\leq e^{\omega_0t}$ for some $\omega_0>0$.
Then both $A$ and its adjoint $A^*$ are quasi-positive in the sense that
\begin{equation}\label{T-eqb.2}
\begin{cases}
\aimininner{ \mathcal Au}{u}_H \geq -\omega_0\aiminnorm{u}_H^2,\hspace{6pt}\forall\, u\in\mathcal{D}(\mathcal A),\\
\aimininner{ \mathcal A^*u}{u}_H \geq -\omega_0\aiminnorm{u}_H^2,\hspace{6pt}\forall\, u\in\mathcal{D}(\mathcal A^*).
\end{cases}
\end{equation}
\end{prop}
\begin{proof}
We set $R(t)=e^{-\omega_0 t}S(t)$ for all $t\geq 0$, and it is clearly that $(S(t))_{t\geq 0}$ is a contraction semigroup on $H$, and it is easy to check that the infinitesimal generator of $(S(t))_{t\geq 0}$ is $-B$, where $B=A+\omega_0$. Proposition \ref{S-propa.3} shows that both $B$ and $B^*$ are positive, which implies that both $A$ and $A^*$ are quasi-positive in the sense of \eqref{T-eqb.2}.
\end{proof}

Let $n$ be a positive integer, and $H^n$ be the direct product of the Hilbert space $H$ with the following scalar product and norm
\begin{equation}\nonumber
\aimininner{x}{y}=\sum_{i=1}^n\aimininner{x_i}{y_i}_H,\quad\quad\aiminnorm{x} = \aimininner{x}{x}^{1/2},
\end{equation}
where $x=(x_1,\cdots,x_n)^t$, $y=(y_1,\cdots,y_n)^t$. Then we have
\begin{thm}\label{T-thm2.5}
Let $(A,\mathcal{D}(A))$ be a closed, densely defined operator on a Hilbert space $H^n$, and assume that $-A$ generates a (resp. quasi-)contraction semigroup on $H^n$, and for any non-singular matrix $P\in GL(n,\mathbb{R})$, we define the operator $B$ by $Bx=P^tAPx$, $\forall\,x\in\mathcal{D}(B)$ and 
\[
\mathcal{D}(B)=\{x\in H^n\,:\,Px\in\mathcal{D}(A)\}.
\]
Then the operator $-B$ also generates a (resp. quasi-)contraction semigroup on $H^n$.
\end{thm}
\begin{proof}
Since $P$ is non-singular, then the map $\mathbb{P}$ induced by $P$ is an isomorphism, where $\mathbb{P} : H^n\mapsto H^n$ given by $\mathbb{P}x= Px$. Hence, it is clear that $(B,\mathcal{D}(B))$ is closed and densely defined. 
If $-A$ generates a (resp. quasi-)contraction semigroup, we notice that $A$ is (resp. quasi-)positive (see \eqref{T-eqb.1} for the meaning of an operator being quasi-positive) by Proposition \ref{S-propa.3} (resp. Proposition \ref{S-prop2.6}), and 
\begin{equation}\begin{split}
\aimininner{Bx}{x} &= \aimininner{P^tAPx}{x} =\aimininner{APx}{Px} \geq 0,\\
\big(\text{resp. }\aimininner{Bx}{x} &= \aimininner{P^tAPx}{x} =\aimininner{APx}{Px} \\
&\geq -\omega_0\aiminnorm{Px}^2 \geq -\omega_0\aiminnorm{P}^2\aiminnorm{x}^2, \text{ for some }\omega_0>0\big),
\end{split}\end{equation}
for all $x\in \mathcal{D}(B)$ (i.e. $Px\in\mathcal{D}(A)$). Hence, $B$ is (resp. quasi-)positive.
We prove in a similar way that $B^*$ is (resp. quasi-)positive. Therefore by Theorem \ref{S-thm2.2} (resp. Theorem \ref{S-thm2.5}), the operator $-B$ also generates a (resp. quasi-)contraction semigroup on $H^n$.
\end{proof}

\section*{Acknowledgments}
This work was partially supported by the National Science Foundation under the grant NSF DMS-1206438 and by the Research Fund of Indiana University.

\bibliographystyle{amsalpha}
\providecommand{\bysame}{\leavevmode\hbox to3em{\hrulefill}\thinspace}
\providecommand{\MR}{\relax\ifhmode\unskip\space\fi MR }
\providecommand{\MRhref}[2]{%
  \href{http://www.ams.org/mathscinet-getitem?mr=#1}{#2}
}
\providecommand{\href}[2]{#2}

\end{document}